\documentclass[11pt, twoside]{article}
\usepackage{textcomp}
\usepackage[utf8]{inputenc}
\usepackage[english]{babel} 
\usepackage[T1]{fontenc} 
\usepackage[utf8]{inputenc} 
\usepackage{soul}
\usepackage[normalem]{ulem}
\usepackage{fancybox}
\usepackage{fancyvrb}
\usepackage{listings}
\usepackage{moreverb}
\usepackage{url}
\usepackage{graphicx}
\graphicspath{ {./images/} }
\usepackage{textcomp}
\usepackage{lmodern}
\usepackage{eurosym}
\usepackage[cyr]{aeguill}
\usepackage{amsmath}
\usepackage{wrapfig} 
\usepackage{makeidx} 
\usepackage{picinpar} 
\usepackage[final]{pdfpages} 
\usepackage{array,multirow,tabularx}
\usepackage{amssymb}
\usepackage[a4paper]{geometry}
\usepackage{amsthm}
\usepackage{tikz}
\usepackage{fancyhdr}
\usepackage{colortbl}
\usepackage{color}
\usepackage{setspace}
\usepackage{longtable}
\usepackage{hhline}
\usepackage{arydshln}
\usepackage{makecell}
\usepackage{mathtools}
\usepackage{tkz-tab}
\usepackage{tikz-cd}
\usepackage{multicol}
\usepackage{enumerate} 
\usepackage{hyperref}
\tikzset{node distance=2.0cm, auto}
\usepackage{fancyhdr}
\renewcommand{\sectionmark}[1]{}
\renewcommand{\subsectionmark}[1]{}

\usepackage[section]{placeins}

\DeclareMathOperator{\ad}{ad}

\allowdisplaybreaks

\title{Rota Baxter operators }
\author{Jon Beristain\footnote{Université de Haute-Alsace, IRIMAS UR 7499, F-68100 Mulhouse, France. \\ E-mail: \texttt{jon.beristain@uha.fr}. }  \; and  Abdenacer Makhlouf\footnote{Université de Haute-Alsace, IRIMAS UR 7499, F-68100 Mulhouse, France.\\ E-mail:  \texttt{abdenacer.makhlouf@uha.fr}.} }
\title{
Restricted (Relative) Rota-Baxter operators on restricted Lie algebras and restricted Lie triple systems and related structures}
\date{\today}

\begin{document}
	
	\newtheorem{thm}{Theorem}[section]
	\newtheorem{prop}[thm]{Proposition}
	\newtheorem{lem}[thm]{Lemma}
	\newtheorem{cor}[thm]{Corollary}
	\theoremstyle{definition}
	\newtheorem{defi}[thm]{Definition}
	\newtheorem*{rmq}{Remark}
	\newtheorem{ex}[thm]{Example}
	\maketitle
	
	\begin{abstract}
		The main purpose of this paper is to define and study  the notion of  restricted Rota-Baxter operator on  restricted Lie algebras and restricted Lie triple systems. We provide the relevant properties and the  usual connections with pre-Lie structures. We show that the operad of pre-Lie triple systems
is the splitting of the operad of Lie triple systems. Moreover, we prove that
Jacobson’s identities hold in pre-Lie triple systems in positive characteristic and then introduce
the notion of restricted pre-Lie triple system.
	\end{abstract}
	
	\noindent\textbf{Keywords:} Restricted Rota-Baxter operator, restricted Lie algebra, restricted Lie triple system, restricted pre-Lie triple system, operad. \\
	\noindent\textbf{MSC 2020 classification:} 		17A40, 17B50,	17B38.  
	\tableofcontents
     \markboth{J. Beristain and A. Makhlouf}{On Rota-Baxter operators in positive characteristic}
     \markright{On Rota-Baxter operators in positive characteristic}
	
\section{Introduction}
Rota-Baxter operators, introduced by G. Baxter in probability theory (\cite{B60}), generalize integration by parts and later reappeared in renormalization theory. They have since found applications in Combinatorics, number theory and operads, see \cite{Rota1995,Guo2000}.

Lie triple systems were introduced by Jacobson in order to study the representations of the Jordan algebras (\cite{J51}). Lie triple systems are related to symmetric spaces in the same way that Lie algebras are related to Lie groups; they are tangent spaces of symmetric spaces. Moreover, Jacobson showed that the Lie triple systems are exactly the subspace of the Lie algebra closed to $[[\cdot,\cdot],\cdot] $. Indeed, a subspace $S$ of a Lie algebra $L$ such that $[[x,y],z] $ lies in $S$ for all $x,y,z$ in $S$ is a Lie triple system for this operation. Moreover, it was shown in \cite{J51} that we can always embedded a Lie triple system $T$ into a Lie algebra $\mathcal{G} $ in such a way that $T$ is a subspace of $\mathcal{G} $ closed under the bracket $[[\cdot,\cdot],\cdot] $. Lie triple systems were strongly studied by Lister in \cite{L51}.

Pre-Lie algebras first appeared in the work of Cayley when he dealt with rooted trees algebras (\cite{C57}). They also appeared in the work of Gerstenhaber on Hochschild cohomology and deformations of algebras (\cite{G63}).

In Lie algebras, Rota-Baxter operators arose independently as operator forms of the classical Yang-Baxter equation (CYBE). The CYBE, central in integrable systems and quantum groups, admits an operator formulation equivalent to the Rota-Baxter relation of weight zero (\cite{STS83}). Rota-Baxter operators are also related to the splitting of algebraic structures, as well as to pre-Lie structures, such as pre-Lie algebras and pre-Lie triple systems, defined by Mabrouk in \cite{M21}. Pre-Lie triple systems play the same role for Lie triple systems as pre-Lie algebras do for Lie algebras. Kupershmidt further generalized these operators into $\mathcal{O}$-operators, also known as relative Rota-Baxter operators. 

In positive characteristic $p>0$, an additional structure arises naturally on certain Lie algebras. This was noticed by Jacobson, who observed that in an associative algebra $A$ over a base field of characteristic $p>0$, for all derivations $D$, $D^p$ is again a derivation (\cite{J37}). Moreover, the Frobenius map $x\in A\mapsto x^p$ is linked with the Lie bracket induced by the Jacobson identities. This led Jacobson to define an abstract $p$-map and the notion of restricted Lie algebras (\cite{J41}). Later, Hodge generalized this notion to Lie triple systems and defined restricted Lie triple systems in \cite{H01}.
In positive characteristic, Dzhumadil'daev showed that Jacobson's identities hold in right-symmetric algebras and thus defined the notion of restricted pre-Lie algebras (\cite{D01}). Later, Dokas gave a more general definition of restricted pre-Lie algebras (\cite{D12}). Ehret and Gilliers extended the results to a larger class. They showed that Jacobson's identities hold for post-Lie algebras and defined the notion of trivially restricted post-Lie algebras (\cite{EG25}).

The aim of this paper is to introduce and study  the notion of restricted Rota-Baxter operators on restricted Lie algebras and restricted Lie triple systems, and to generalize their usual connection with pre-Lie structures. We show that the operad of pre-Lie triple systems is the splitting of the operad of Lie triple systems considered here. Moreover, we prove that Jacobson's identities hold in pre-Lie triple systems in positive characteristic and then introduce the notion of restricted pre-Lie triple systems. Throughout this paper, the characteristic of the base field is assumed to be different from two. 

This paper is organized as follows. In section \eqref{s2}, we recall the definition of Rota-Baxter operators on Lie algebras and their connection with pre-Lie algebras. In section \eqref{s3}, we define the notion of restricted Rota-Baxter operators on restricted Lie algebras in \hyperref[d3.6]{Definition~\ref*{d3.6}} and show that similar relations hold with restricted pre-Lie algebras (\hyperref[p3.13]{Proposition~\ref*{p3.13}}). Section \eqref{s4} is devoted to Rota-Baxter operators on Lie triple systems. We recall the relevant definitions and their connection with pre-Lie triple systems. In subsection \eqref{sub4}, we justify the definition of pre-Lie triple systems using operad theory. We first recall some facts about operads, then define the operad of Lie triple systems and finally apply the method developed in \cite{PBG13} to split this operad. In section \eqref{s5}, we consider restricted Lie triple systems. We introduce the notion of restricted Rota-Baxter operators on restricted Lie triple systems (\hyperref[d5.6]{Definition~\ref*{d5.6}}). We prove that Jacobson's identities hold in pre-Lie triple systems in positive characteristic (\hyperref[T5.9]{Theorem~\ref*{T5.9}}), which allows us to define restricted pre-Lie triple systems (\hyperref[D5.14]{Definition~\ref*{D5.14}}) and to establish results analogous to those in the binary case (\hyperref[p5.20]{Proposition~\ref*{p5.20}}). Finally, in section \eqref{s6}, we extend these results to the setting of $\mathcal{O} $-operators.
\section{Rota-Baxter operators on Lie algebras}\label{s2}

In this section, we recall the definition of Rota-Baxter operators and some relevant results in this topic.

\subsection{Definitions}

\begin{defi}
    Let $L$ be a Lie algebra. A Rota-Baxter operator of weight $0$ on $L$ is a linear map $P:L\to L$ such that $$[Px,Py]=P([Px,y] +[x,Py]) $$ for all $x,y\in L$. A Lie algebra $L$ equipped with a Rota-Baxter operator $P$ is denoted by $(L,P)$, and called a Rota-Baxter Lie algebra.
\end{defi}

    An important characterization of Rota-Baxter operators is given through graphs. We have the following  equivalent definition:  a linear map $P:L\to L$ is a Rota-Baxter operator on $L$, if and only if, its graph $$Gr(P)=\left\{ (Px,x), x\in L\right\}\subset L\oplus L $$ is a Lie subalgebra of the semi-direct product $L\oplus L$ of $L$ with the $L$-module given by the adjoint representation.

In the sequel, we are going to provide  a restricted analog of  Rota-Baxter operator and this equivalence; see \hyperref[p3.10]{Proposition~\ref*{p3.10}}.

\subsection{Connection with pre-Lie algebras}

\begin{defi}
    A pre-Lie algebra is a pair $(A,\circ) $, where $A$ is a vector space equipped with a binary operation $\circ : A\times A\to A $ such that the following identity holds: \begin{equation}
        x\circ (y\circ z)- (x\circ y)\circ z= y\circ(x\circ z) - (y\circ x)\circ z ,  \quad  \forall x,y,z\in A.
    \end{equation}
        
 Setting  $a(x,y,z)= x\circ (y\circ z) - (x\circ y)\circ z $, the associator,  then the above identity can be written as \begin{equation}\label{eq 1}
        a(x,y,z)=a(y,x,z) .
    \end{equation} For this reason, pre-Lie algebras are also called left-symmetric algebras. 
\end{defi}

\begin{ex}
    \begin{enumerate}
        \item An associative algebra is a pre-Lie algebra.
        \item Pre-Lie algebras appeared in Gerstenhaber's construction of the Hochschild cohomology of associative algebras, see (\cite{G63}). 
    \end{enumerate}
\end{ex}

Pre-Lie algebras form an important class of  Lie admissible algebras.

\begin{prop}
    Let $(A,\circ) $ be a pre-Lie algebra. Then $A$ is endowed with a Lie algebra structure, where the  bracket is defined by $[x,y]:= x\circ y - y\circ x $, for all $x,y\in A$. Moreover, the left multiplication $L:x\mapsto (y\mapsto x\circ y) $ defines a representation of the induced Lie algebra.
\end{prop}
Conversely, one may obtain a pre-Lie algebra starting from a Rota-Baxter Lie algebra.
\begin{prop}
    Let $(L,P)$ be a Rota-Baxter Lie algebra. Then, the product $x\circ y = [Px,y] $ provides a pre-Lie algebra structure on $L$. 
\end{prop}

\begin{rmq}
    A nonassociative  $(A,\circ)$ is right-symmetric  algebra if  the associator is symmetric in the right two variables $y,z$,  that is,  $
        a(x,y,z)= a(x,z,y).$
   
    Note that a vector space with a binary operation $(A,\circ)$ is left-symmetric if and only if its opposite algebra $(A,\circ^{opp})$ is right-symmetric. In this case, the Lie brackets induced by these two algebras are opposite, ie $[x,y]^{opp}=-[x,y] $ for all $x,y\in A$.
\end{rmq}

Let $(L,P)$ be a Rota-Baxter Lie algebra. Defining $x\circ y= [Px,y] $, we obtain a pre-Lie algebra $(L,P)$. The Lie algebra associated with this pre-Lie structure has bracket $[x,y]_{C}=x\circ y-y\circ x $. Then $P:(L,[\cdot,\cdot]_C)\to (L,[\cdot,\cdot]) $ is a morphism of Lie algebras. 
\[\begin{tikzcd}[row sep=large, column sep=large]
	{(L,[\cdot,\cdot])} & {(L,\circ)} \\
	& {(L,[\cdot,\cdot]_C)}
	\arrow["{[Px,y]}", from=1-1, to=1-2]
	\arrow[from=1-2, to=2-2]
	\arrow["P", from=2-2, to=1-1]
\end{tikzcd}\]

\section{Rota-Baxter operators on restricted Lie algebras}\label{s3}

\subsection{Restricted Rota-Baxter operators}

 In this subsection, we recall some results on restricted Lie algebras and introduce the notion of restricted Rota-Baxter operators.

\begin{defi}\label{d3.1}
    We say that $L$ is a restricted Lie algebra over a field $\mathbf{K}$ of characteristic $p>0$ if $L$ is a Lie algebra equipped with a $p$-map $(\cdot)^{[p]}: L\to L$ satisfying the following properties: \begin{enumerate}
        
        \item $(\alpha x)^{[p]}= \alpha^p x^{[p]} $, for all $\alpha \in \mathbf{K}$ and $x\in L$,
        \item  $[x^{[p]},y]= (ad_x)^p(y) $, for all $x\in L$ and $y\in L $,
        \item  $(x+y)^{[p]}= x^{[p]} + y^{[p]} + \sum\limits_{i=1}^{p-1} s_i(x,y) $, for all $ x,y\in L$,
    \end{enumerate} where $\ad_x(y) = [x,y] $ and $is_i(x,y) $ is the coefficient of $\lambda^{i-1} $ in $(\ad_{\lambda x+y})^{p-1}(x) $.
\end{defi}

\begin{defi}
    Let $V$ be a vector space. We say that a map $f:V\to V $ is $p$-semi linear if $f(\alpha x+ y)= \alpha^p f(x)+ f(y) $ for all $x,y\in V$, $\alpha\in \mathbf{K}$.
\end{defi}

\begin{defi}
    Let $(L,(\cdot)^{[p]}), (L',(\cdot)^{[p]'})$ be two restricted Lie algebras. A morphism of restricted Lie algebras between $L$ and $L'$ is a morphism of Lie algebras $f:L\to L' $ such that $f(x^{[p]})=f(x)^{[p]'} $ for all $x\in L$.
\end{defi}

\begin{ex}
    \begin{enumerate}
        \item Let $A$ be an associative algebra. Then the associated Lie algebra is restricted by defining $x^{[p]}=x^p $ for all $x\in A$.
        \item Let $L$ be an abelian Lie algebra with a map $f:L\to L $. Then, $(L,f)$ is restricted if and only if $f $ is $p$-semi linear.
    \end{enumerate}
\end{ex}

\begin{rmq}
    Even if the $p$-map is not linear, points $1$ and $3$ of the \hyperref[d3.1]{Definition~\ref*{d3.1}} show that it is determined by its value on a basis. Moreover, if $(\cdot)^{[p]_1} $ and $(\cdot)^{[p]_2} $ are two $p$-map of a restricted Lie algebra $L$, then $(\cdot)^{[p]_1}-(\cdot)^{[p]_2} $ takes value in the center $C(L) $ of $L$. Then, $L$ being centerless implies that the $p$-map, if it exists, is unique.
\end{rmq}

The following result is useful in order to show the existence of a restricted structure on a Lie algebra.

\begin{thm}[Jacobson]\label{thm3.4}
    Let $L$ be a Lie algebra. Assume that there exists $(e_{j})_{j\in J} $ a basis of $L$ such that there are $y_{j}\in L_{p} $ with $(\ad (e_{j}))^p= \ad (y_{j})$ for all $j\in J $. Then there exists exactly one $p$-map $(\cdot)^{[p]}:L\to L $ such that $e_{j}^{[p]}=y_{j} $ for all $j\in J$. 
\end{thm}

\begin{defi}
    A restricted representation of a restricted Lie algebra $L$ is a morphism of restricted Lie algebras $\rho:L\to \mathrm{End}(V)_G$, where $V$ is a vector space, $\mathrm{End}(V)$ is endowed with the Lie bracket $[f,g]=fg-gf $ and the $p$-map is given by the $p^{th}$ power.
\end{defi}

\begin{ex}
    The adjoint representation of a restricted Lie algebra is a restricted representation.
\end{ex}

Let $L$ be a restricted Lie algebra and let $\varphi:L\to \mathrm{End}(V)$ be a restricted representation. We define a Lie algebra structure on $L\oplus V $ with the bracket \begin{center}
    $[(x,v),(y,w)]= ([x,y],\varphi(x)(w) - \varphi(y)(v)) $ for all $(x,v),(y,w)\in L\oplus V $.
\end{center}

\begin{prop}[\cite{BEMS25}]\label{p3.7}
    Let $L$ be a restricted Lie algebra and let $(\varphi,V)$ be a restricted representation of $L$. Then, for any basis $(x_i,v_i)$ of $L\oplus V$, there exist a $p$-map on $L\oplus V$ given by the formula $$(x_i,v_i)^{[p]}= (x_i^{[p]}, \varphi(x_i)^{p-1}(v_i)) $$ on the basis. 
\end{prop}

\begin{rmq}
    The proof of this proposition uses Jacobson's theorem. In fact, it is proved in \cite{BEMS25} that for all $(x,u),(y,v) $ in $L\oplus V$, $[(x^{[p]},\varphi(x)^{p-1}(u)),(y,v)]= \ad_{(x,u)}^{p}(y,v) $. Thus, each choice of a basis of $L\oplus V $ gives rise to a $p$-map by Jacobson's theorem. If $L\oplus V$ is centerless, then all these $p$-maps coincide and the map $(x,u)\mapsto (x^{[p]},\varphi(x)^{p-1}(u)) $ is the unique $p$-map on $L\oplus V $.
\end{rmq}

We now define the notion of restricted Rota-Baxter operator on a restricted Lie algebra.

\begin{defi}\label{d3.6}
    Let $L$ be a restricted Lie algebra. A Rota-Baxter operator $P$ is said to be restricted if $$P(x)^{[p]}=P(\ad_{Px}^{p-1}(x)) $$ for all $x\in L $.
\end{defi}

This definition is motivated by the following result.

\begin{prop}\label{p3.10}
    Let $L$ be a restricted Lie algebra. A Rota-Baxter operator $P$ is restricted if, and only if, $Gr(P) $ is a restricted Lie subalgebra of $L\oplus L$, where $L\oplus L $ is endowed with the $p$-map defined in the \hyperref[p3.7]{Proposition~\ref*{p3.7}} with respect to the adjoint representation, for any basis of $L\oplus L $. 
\end{prop}

\begin{proof}
    It remains to show that $Gr(P) $ is closed under the $p$-map in $L\oplus L $ if and only if $P$ is restricted. First, we observe that for $Gr(P) $ to be closed under a $p$-map it is sufficient that for a basis $(Px_i,x_i) $ of $Gr(P)$, $(Px_i,x_i)^{[p]}\in Gr(P) $. The $p$-map defined on the semi-direct product $L\oplus L$ with respect to the adjoint representation is given, on a basis $(x_i,y_i) $, by $(x_i,y_i)^{[p]}=(x_i^{[p]},\ad_{x_i}^{p-1}(y_i)) $. We can choose a basis of $L\oplus L$ by taking a basis $(Px_i,x_i) $ of $Gr(P)$ and completing it to a basis of $L\oplus L$. Therefore, a Rota-Baxter operator must satisfy the condition $P(x_i)^{[p]}=P(\ad_{Px_i}^{p-1}(x_i)) $. In this case, $(Px_i,x_i)^{[p]}\in Gr(P)$ and $Gr(P) $ is a restricted subalgebra of the restricted Lie algebra $L\oplus L$. This condition is exactly the definition of a restricted Rota-Baxter operator. Conversely, assume that $Gr(P) $ is stable under any $p$-map induced by a choice of basis of $L\oplus L$. Let $x\in L$ be a nonzero element. Then $(Px,x) $ can be completed to a basis of $L\oplus L $. Since $Gr(P)$ is stable under the corresponding $p$-map, we have $(Px,x)^{[^p]}=(P(x)^{[p]},\ad_{Px}^{p-1}(x))\in Gr(P) $ which shows that $P $ is a restricted Rota-Baxter operator.
\end{proof}

\begin{rmq} Assume that $L $ is centerless. Then the semi-direct product $L\oplus L $ is also centerless. The map $(x,y)\mapsto (x^{[p]},\ad_x^{p-1}(y)) $ is then a $p$-map. 
    Thus, the condition on the graph of $P$ shows that it suffices to verify that $P(e_i)^{[p]}=P(\ad_{Pe_i}^{p-1}(e_i)) $ for a basis $(e_i)$ of $L$ to prove that $P$ is restricted. Indeed, in order to show that $Gr(P) $ is closed under the $p$-map, it suffices to show that for a generating family $(u_i)$ of $Gr(P)$, we have $u_i^{[p]}\in Gr(P) $. However, $((Pe_i,e_i)) $ is a generated family of $Gr(P) $.
\end{rmq}

\begin{ex}
We recall that the restricted Witt algebra is the $\mathbf{K}- $vector space $W=\bigoplus\limits_{i=-1}^{p-2}\mathbf{K}e_i $ endowed with the Lie bracket defined by $$[e_i,e_j]=(j-i)e_{i+j}, $$ where $j-i $ and $i+j $ are computed modulo $p$ and the $p$-map is defined by \begin{center}
    $e_0^{[p]}=e_0$ and $e_i^{[p]}=0 $ for $i= -1,1,2,\ldots, p-2$.
\end{center}  
   Using the work of \cite{GLBJ16}, the operator $P_k^\alpha$ defined by $P_k^\alpha(e_m)=\alpha\delta_{m+2k,0}e_{m+k} $ with $\alpha\in \mathbf{K},k\in (\mathbf{Z}/p\mathbf{Z})^\star $ is a Rota-Baxter operator on $W$. It is a restricted Rota-Baxter operator. Indeed, let $x= \sum\limits_{i=-1}^{p-2}x_i e_i\in W $. We want to show that $$P_k^\alpha(x)^{[p]}= P_k^\alpha(\ad_{P_k^\alpha(x)}^{p-1}(x)). $$ We have that $$P_k^\alpha(x)= x_{-2k}P_k^\alpha(e_{-2k})= \alpha x_{-2k}e_{-k} $$ so $P_k^\alpha(x)^{[p]}=0 $. To obtain the result, it remains to show that the projection of $\ad_{P_k^\alpha(x)}^{p-1}(x) $ on $\mathbf{K}e_{-2k} $ is zero. However, $$[e_{-k},x]= \sum\limits_{i\neq -k}(i+k) x_ie_{i-k} ,\  [e_{-k},[e_{-k},x]]=\sum\limits_{i\neq -k,0}i(i+k)x_ie_{i-2k}, $$ and finally, in $\ad_{P_k^\alpha(x)}^{p-1}(x) $ the coefficient of $e_{i+k}$ vanishes for $i=-k,0,\ldots,(p-3)k$ and the projection of $\ad_{P_k^\alpha(x)}^{p-1}(x) $ on $\mathbf{K}e_{-2k} $ is zero. So $P_k^\alpha(\ad_{P_k^\alpha(x)}^{p-1}(x))=0 $ and $P_k^\alpha$ is a restricted Rota-Baxter operator.
\end{ex}

\subsection{Restricted pre-Lie algebras}

Dzhumadil'daev showed that, in right-symmetric algebras, Jacobson's formula holds for the $p^{th}$ power of the sum of two elements. More precisely, we have the following result.

\begin{prop}[\cite{D01}]
    Let $(A,\circ) $ be a right symmetric algebra. Then, for any $x,y\in A$, we have \begin{center}
        $(x+y)^{.p}= x^{.p}+ y^{.p} + \sum\limits_{i=1}^{p-1} s_i(x,y) $,
    \end{center} where $is_i(x,y) $ is the coefficient of $t^{i-1} $ in $\ad(tx+y)(x) $ and $$x^{.p}= ((((x\circ x)\circ x)\cdots)\circ x)\circ x. $$
\end{prop}

 Let $(A,\circ) $ be a right symmetric algebra. Then $(A,\circ^{opp}) $ is a pre-Lie algebra and the Lie brackets induced by $(A,\circ) $ and $(A,\circ^{opp}) $ are opposite, ie $[x,y]^{opp}=-[x,y] $ for all $x,y\in A$. In that case, the elements $s_i(x,y) $ coincide, and Jacobson's formula remains valid for pre-Lie algebras. This result justifies the following definition. Dzhumadil'daev introduced a notion of restricted pre-Lie algebra which we will call a trivially restricted pre-Lie algebra, following the terminolgy of \cite{EG25}.

\begin{defi}[\cite{D01}]\label{D3.7}
    Let $(L,\circ)$ be a pre-Lie algebra over a field $\mathbf{K}$ of characteristic $p$. We say that $L$ is a trivially restricted pre-Lie algebra if the following relation holds: \begin{equation}\label{eq3}
        x^{\circ. p}\circ y = x\circ (x\circ (\cdots (x\circ y))) 
    \end{equation} where \begin{equation}
        x^{\circ .p}=x\circ (x\circ (\cdots (x\circ x)))  .
    \end{equation}
\end{defi}

\begin{prop}[\cite{D01}]\label{p3.8}
    Let $(A,\circ) $ be a trivially restricted right-symmetric algebra. Then the Lie algebra induced by $(A,\circ)$ is restricted, with $x^{[p]}= x^{.p} $. 
\end{prop}

Dokas gave in \cite{D12} a generalization of the notion of restricted pre-Lie algebra by introducing an abstract $p$-map.

\begin{defi}[\cite{D12}]
    A restricted pre-Lie algebra $(L,\circ,(\cdot)^{[p]}) $, is a pre-Lie algebra over a field $\mathbf{K} $ of characteristic $p>0$, together with a map $(\cdot)^{[p]}:L\to L $, called the $p$-map, such that: \begin{enumerate}
        \item $(\alpha x)^{[p]}= \alpha ^px^{[p]} $
        \item  $x^{[p]}\circ y = (x\circ (x\circ (\cdots (x\circ y)))) $
        \item $y\circ x^{[p]}= (x\circ (x\circ (\cdots (x\circ y)))) - [x,[x,[\cdots[x,y]]]] $
        \item $(x+y)^{[p]}= x^{[p]}+y^{[p]}+ \sum\limits_{i=1}^{p-1} s_i(x,y)$
    \end{enumerate} where $[x,y]=x\circ y - y\circ x $ denotes the induced bracket.
\end{defi}

From \hyperref[D3.7]{Definition~\ref*{D3.7}} and \hyperref[p3.8]{Proposition~\ref*{p3.8}}, we see that a trivially restricted pre-Lie algebra $A$ is restricted in the sense of Dokas with the $p$-map given by $x^{[p]}=x^{\circ.p} $ for all $x\in A$.

\begin{prop}
    Let $(A,\circ, (\cdot)^{[p]}) $ be a restricted pre-Lie algebra. Then, the left multiplication $L:A\to \mathrm{End}(A), x\mapsto (y\mapsto x\circ y) $ is a restricted representation of the induced restricted Lie algebra.
\end{prop}

\begin{proof}
    We have to show that for all $x\in A $, $L(x^{[p]})=L(x)^p $. It comes from point $2.$ of the definition of a restricted pre-Lie algebra.
\end{proof}

\subsection{Connection with restricted pre-Lie algebras}

If $(L,\circ,(\cdot)^{[p]}) $ is a restricted pre-Lie algebra, then the commutator bracket $[\cdot,\cdot]_C $ endows $L$ with the structure of a restricted Lie algebra $(L,[\cdot,\cdot]_C,(\cdot)^{[p]}) $, with the same $p$-map. 

\begin{prop}\label{p3.13}
Let $( A,[\cdot ,\cdot ], (\cdot)^{[p]}) $ be a restricted Lie algebra endowed with a restricted Rota-Baxter operator $P$. Define an operation $\circ$ on $A$ by
\begin{equation}
x\circ y=[P(x), y].
\end{equation}
Then $( A, \circ) $ is a trivially restricted pre-Lie algebra, that is, it is a restricted pre-Lie algebra with \begin{center}
    $x^{[p]'}= ad_{Px}^{p-1}(x)  $
\end{center} as $p$-map. $P$ is then a morphism of restricted Lie algebras $$P: ( A,[\cdot ,\cdot ]_C, (\cdot)^{[p]'}) \to (A,[\cdot,\cdot],(\cdot)^{[p]}).  $$
\end{prop}

\begin{proof}
    First, we note that $x^{[p]'}= x^{\circ.p} $. So we just have to verify that equation \eqref{eq3} holds in $(A,\circ)$. For $x,y\in A$, we have \begin{equation}\label{eq4}
        x^{\circ. p}\circ y = [P(\ad_{Px}^{p-1}(x)),y]= [P(x)^{[p]},y]= \ad_{Px}^p(y)=(x\circ (x\circ (\cdots (x\circ y)))). \end{equation}  Thus, $(A,\circ) $ is trivially restricted. The fact that $P$ is a morphism of restricted Lie algebras follows directly from the \hyperref[d3.6]{Definition~\ref*{d3.6}} and the expression of the $p$-map. 
\end{proof}

From a restricted Lie algebra $(L,[\cdot,\cdot],(\cdot)^{[p]})$ endowed with a restricted Rota-Baxter operator $P$, we obtain a structure of restricted pre-Lie algebra on $L$ by defining $x\circ y=[P(x), y]$ and $x^{[p]'}= ad_{Px}^{p-1}(x)  $. The Lie algebra induced by this restricted pre-Lie algebra is itself restricted, and the Rota-Baxter operator $P$ is a morphism of restricted Lie algebras $P:(L,[\cdot,\cdot]_C,(\cdot)^{\circ.p} )\to (L,[\cdot,\cdot],(\cdot)^{[p]}) $. This situation can be summarized by the following diagram:

\[\begin{tikzcd}[row sep=large, column sep=large]
	{(L,[\cdot,\cdot],(\cdot)^{[p]})} & {(L,\circ, (\cdot)^{\circ.p})} \\
	& {(L,[\cdot,\cdot]_C,(\cdot)^{\circ.p})}
	\arrow["{[Px,y]}", from=1-1, to=1-2]
	\arrow[from=1-2, to=2-2]
	\arrow["P", from=2-2, to=1-1]
\end{tikzcd}\]




\section{Rota-Baxter operators on Lie triple systems}\label{s4}
In this section, we give some basics about Lie triple systems, define Rota-Baxter operators on Lie triple systems and present some related structures. Precisely, we show new relationships with pre-Lie triple systems and provide a new interpretation in terms of operads. 
\subsection{Definitions}

In this subsection, we provide some basic  definitions about  Lie triple systems and Rota-Baxter operators on Lie triple systems. Moreover, we give some relevant results.

\begin{defi}
    A Lie triple system is a pair $(T,[\cdot,\cdot,\cdot]) $ consisting of a vector space $T$ and a trilinear map $[\cdot,\cdot,\cdot]:T\times T\times T\to T $ such that for all $u,v,x,y,z\in T$, the following identities are satisfied: \begin{enumerate}
        \item  $[x,y,z]= -[y,x,z] $,
        \item  $[x,y,z]+ [y,z,x]+[z,x,y]=0 $,
        \item $[u,v,[x,y,z]]=[[u,v,x],y,z]+ [x,[u,v,y],z] + [x,y,[u,v,z]] $. 
    \end{enumerate} 
\end{defi}

\begin{ex}
     Let $L$ be a Lie algebra. Then any subspace $T$ of $L$ closed under the operation $(x,y,z) \mapsto[[x,y],z] $, for all $x,y,z \in T $, is a Lie triple system denoted by $L_{trip} $.
\end{ex}

\begin{defi}
    Let $T$ be a Lie triple system and $H $ be a subspace of $T$. Then $H$ is called a subsystem of $T$ if $[H,H,H]\subset H $.
\end{defi}

\begin{defi}
    Let $T$ be a Lie triple system. An embedding of $T$ is a Lie algebra together with a linear map $\varphi: T\to L $ such that, for all $x,y,z\in T $, $$\varphi([x,y,z])= [[\varphi(x),\varphi(y)],\varphi(z)]. $$
\end{defi}

\begin{ex}
    Jacobson showed in \cite{J51} that for every Lie triple systems $T$, there exists an injective embedding. there are two important examples of embedding. The standard embedding $L_S(T)$ and the universal embedding $L_U(T) $. They are both injective embedding so $T$ can be viewed as a subspace of $L_S(T)$ and $L_U(T)$ closed under the operation $[[\cdot,\cdot],\cdot] $.
\end{ex}

\begin{defi}
    Let $T$ be a Lie triple system. A linear map $P:T\to T$ is called a Rota-Baxter operator of weight $0$ on $T$ if $$[Px,Py,Pz]= P([x,Py,Pz]+[Px,y,Pz]+[Px,Py,z]) $$ for all $x,y,z \in T$.
\end{defi}

\begin{ex}
    Let $(L,[\cdot,\cdot])$ be a Lie algebra endowed with a Rota-Baxter operator $P$. Then $P$ is a Rota-Baxter operator of the induced Lie triple system $(L,[\cdot,\cdot,\cdot]) $.
\end{ex}

\begin{ex}\label{exa4.7}
    Let $T$ be a $2$-dimensional Lie triple system defined with to a basis $\{e_1,e_2\}$ by the bracket defined by $[e_1,e_2,e_2]=e_1, [e_2,e_1,e_2]=-e_1 $ and the other brackets are zero, see \cite{CHMM22}. The operator given by a matrix $P=\begin{pmatrix}
0 & a \\
0 & b \\
\end{pmatrix} $ is a Rota-Baxter operator.
\end{ex}

\begin{defi}
    Let $T$ be a Lie triple system, $V$ be a $\mathbf{K}$-vector space and $\theta:T\times T\to \mathrm{End}(V) $ be a bilinear map. A pair $(V,\theta)$ is called a representation of $T$, if for all $x,y,z,t\in T$, the following identities hold:
           \begin{enumerate}
           \item 
               $\theta(z,t)\theta(x,y)-\theta(y,t)\theta(x,z)-\theta(x,[y,z,t]) + D(y,z)\theta(x,t)=0$,
           \item $\theta(z,t)D(x,y)-D(x,y)\theta(z,t)+\theta([x,y,z],t)+ \theta(z,[x,y,t])=0$ ,
          \end{enumerate} where $D(x,y)=\theta(y,x)-\theta(x,y) $.
\end{defi}

\begin{ex}
    The pair $(\theta,T)$, where $\theta(x,y)(z)=[z,x,y] $, is a representation, called the adjoint representation.
\end{ex}

\begin{prop}
    Let $(V,\theta)$ be a representation of a Lie triple system $T$. Consider the operation $[\cdot,\cdot,\cdot]_V: (T\oplus V)\times (T\oplus V)\times (T\oplus V)\to T\oplus V  $ defined by \begin{center}
        $[(x,u),(y,v),(z,w)]_V=([x,y,z], \theta(y,z)(u)-\theta(x,z)(v)+ D(x,y)(w)) $.
    \end{center} Then, $T\oplus V$ is a Lie triple system, called the semi-direct product of $T$ by the representation $(V,\theta)$.
\end{prop}

Specializing the above result with $(\theta,T) $ being the adjoint representation, we obtain the following result. 

\begin{prop}
    Let $T$ be a Lie triple system. A linear map $P:T\to T$ is a Rota-Baxter operator if and only if $Gr(P)=  \left\{ (Px,x), x\in T\right\}  $ is a subsystem of $T\oplus T$, the semi-direct product of $T$ with its adjoint representation.
\end{prop}

\subsection{Connection with pre-Lie triple systems}

In \cite{M21}, Mabrouk defined a notion of pre-Lie triple system in order to extend the notion of pre-Lie algebra to the setting of triple system. As in the binary case, there is still a connection with Lie triple systems.

\begin{defi}[\cite{M21}]\label{d4.6}
    A pre-Lie triple system is a vector space $A$ equipped with a trilinear map $  \left\{ \cdot,\cdot,\cdot\right\}: A\times A\times A \to A $ such that the following identities hold: \begin{equation}\label{equ7}
        \left\{ x_5,x_1, [x_2,x_3,x_4 ]_C\right\}= \left\{\left\{ x_5,x_1,x_2\right\},x_3,x_4 \right\} -  \left\{\left\{ x_5,x_1,x_3\right\},x_2,x_4 \right\} + \left\{x_2,x_3,\left\{ x_5,x_1,x_4\right\}\right\}^{\wedge} ,
\end{equation}
\begin{equation}\label{equ8}
        \left\{x_2,x_3,\left\{ x_5,x_1,x_4\right\}\right\}^{\wedge}= \left\{\left\{ x_1,x_2,x_5\right\}^{\wedge},x_3,x_4 \right\} +\left\{ x_5, [x_1,x_2,x_3 ]_C,x_4\right\} +\left\{ x_5,x_3, [x_1,x_2,x_4 ]_C\right\} ,
    \end{equation} where $ \left\{ \cdot,\cdot,\cdot\right\}^{\wedge}$ and $[\cdot,\cdot,\cdot]_C $ are defined by \begin{align*}
        &\left\{ x,y,z\right\}^{\wedge}= \left\{ z,y,x\right\}- \left\{ z,x,y\right\} \\
& [x,y,z]_C= \left\{ x,y,z\right\} -\left\{y,x,z\right\}+\left\{z,y,x\right\}-\left\{ z,x,y\right\}= \left\{x,y,z\right\}^{\wedge} + \left\{x,y,z\right\}-\left\{ y,x,z\right\} 
    \end{align*} for all $x,y,z,x_i\in A$, $1\leq i\leq 5$.
\end{defi}

\begin{ex}[\cite{M21}]\label{ex4.7}
Let A be a $3$-dimensional vector space generated by $e_1,e_2,e_3  $ equipped with a ternary bracket $\left\{\cdot,\cdot,\cdot \right\} $ defined by $\left\{e_3,e_3,e_3 \right\}=e_2$ and all the other brackets are zero. Then $(A,\left\{\cdot,\cdot,\cdot \right\}) $ is a pre-Lie triple system. 
\end{ex}

\begin{ex}[\cite{M21}]
Let A be a $4$-dimensional vector space generated by $ e_1,e_2,e_3,e_4 $ equipped with a ternary bracket $\left\{\cdot,\cdot,\cdot \right\} $ defined by $$\left\{e_4,e_2,e_3 \right\}=e_2, \left\{e_4,e_2,e_2 \right\}=e_1, \left\{e_3,e_2,e_4 \right\}=\left\{e_4,e_4,e_4 \right\}= -e_2$$ and all the other brackets are zero. Then $(A,\left\{\cdot,\cdot,\cdot \right\}) $ is a pre-Lie triple system.
\end{ex}

\begin{prop}[\cite{M21}]
    Let $(A,\left\{ \cdot,\cdot,\cdot\right\}) $ be a pre-Lie triple system. Then $(A,[\cdot,\cdot,\cdot]_C) $ is a Lie triple system. 
\end{prop}

\begin{prop}
    Let $(A,\left\{ \cdot,\cdot,\cdot\right\}) $ be a pre-Lie triple system. The bilinear map $\theta:T\times T\to \mathrm{End}(T)$ defined by $\theta(x,y)(z) = \left\{ z,x,y\right\} $ is a representation of the induced Lie triple system.
\end{prop}

\begin{proof}
    Equations \eqref{equ7} and \eqref{equ8} are exactly the condition for $\theta $ to be a representation of $(A,[\cdot,\cdot,\cdot]_C) $.
\end{proof}

\begin{prop}[\cite{M21}]
    Let $(A,\circ)$ be a pre-Lie algebra. Define a ternary operation $\left\{\cdot,\cdot,\cdot\right\}: A\times A\times A\to A $ on $A$ by \begin{center}
        $\left\{ x,y,z\right\}= z\circ (y\circ x) $, $\forall x,y,z\in A $.
    \end{center} Then $(A,\left\{\cdot,\cdot,\cdot\right\}) $ is a pre-Lie triple system.
\end{prop}

The situation can be summarized in the following diagram, which is commutative. \[\begin{tikzcd}[row sep=large, column sep=large]
	{\text{pre-Lie}} && \text{Lie} \\
	{ \text{pre-Lie triple system} } && \text{Lts}
	\arrow["{[x,y]=x\circ y - y\circ x}", from=1-1, to=1-3]
	\arrow["{\left\{ x,y,z\right\}= z\circ (y\circ x)}"', from=1-1, to=2-1]
	\arrow["{[x,y,z]=[[x,y],z]}", from=1-3, to=2-3]
	\arrow["{[x,y,x]_C}"', from=2-1, to=2-3]
\end{tikzcd}\]

\begin{prop}[\cite{M21}]
    Let $(T,[\cdot,\cdot,\cdot],P)$ be a Lie triple system endowed with a Rota-Baxter operator. Then one can define on $T$ a pre-Lie triple system structure by setting $\left\{x,y,z \right\}=[x,Py,Pz] $.
\end{prop}

The situation is as follows. From a Lie triple system $T$ endowed with a Rota-Baxter operator $P$ we obtain a structure of pre-Lie triple system on $T$ and then the map $P: (T,[\cdot,\cdot,\cdot]_C)\to (T,[\cdot,\cdot,\cdot]) $ is a morphism of Lie triple systems. 

\[\begin{tikzcd}[row sep=large, column sep=large]
	{(T,[\cdot,\cdot,\cdot])} & {(T,\left\{\cdot,\cdot,\cdot \right\})} \\
	& {(L,[\cdot,\cdot,\cdot]_C)}
	\arrow["{[x,Py,Pz]}", from=1-1, to=1-2]
	\arrow[from=1-2, to=2-2]
	\arrow["P", from=2-2, to=1-1]
\end{tikzcd}\]

\subsection{The operad of pre-Lie triple systems is the splitting of the operad of Lie triple systems}\label{sub4}

Here we give another justification and interpretation for the definition of pre-Lie triple systems from operad theory perspective. We use the method developed in \cite{PBG13} to split the operad of Lie triple systems which we define below. 

\subsubsection{Basics on operads}

We provide  here some basics and constructions about the operads. For a general theory of operads, the reader is referred  to \cite{LV12}. In order to describe and study a type of algebra, instead of considering the elementary operations defining this type of algebra and the relations between these operations, we consider all the operations on a finite number of variables of such an algebra and all the relations between these operations. A convenient way to view elements of an operad is as operations in the corresponding type of algebra. 

    A non-symmetric operad is a collection $(\mathcal{P}(n))_{n\in \mathbf{N}} $ of vector spaces and, for all $k\in \mathbf{N}, n_1,\ldots,n_k\in \mathbf{N} $,   $(k+1) $-linear maps $\circ$ : $$\circ: \left\{\begin{matrix}
\mathcal{P}(k)\times \mathcal{P} (n_1)\times \cdots \times \mathcal{P}(n_k)\longrightarrow \mathcal{P}(n_1+\cdots+ n_k) \\ (p,p_1,\ldots,p_k)\longmapsto p\circ (p_1,\ldots,p_k)
\end{matrix}\right. $$ satisfying the following properties: \begin{enumerate}
    \item For all $k\in \mathbf{N},n_1,\ldots,n_k\in \mathbf{N},n_i^j\in \mathbf{N},(1\leq i\leq k,1\leq j\leq n_i) $ \\ and for all $p\in \mathcal{P}(k),p_i\in \mathcal{P}(n_i), p_i^j\in \mathcal{P}(n_i^j):  $ $$p\circ (p_1\circ (p_1^1,\ldots,p_1^{n_1}),\ldots,p_k\circ(p_k^1,\ldots,p_k^{n_k}))=(p\circ (p_1,\ldots,p_k))\circ (p_1^1,\ldots,p_1^{n_1},\ldots,p_k^1,\ldots,p_k^{n_k}) $$
    \item There exists $I\in \mathcal{P}(1) $ such that, for all $n\in \mathbf{N}, p\in \mathcal{P}(n)$: $$p\circ (I,\ldots,I)=p, \ I\circ p=p. $$
\end{enumerate}

An operad is a non-symmetric operad $\mathcal{P} $ such that, for all $n\in \mathbf{N}$, $\mathcal{P}(n)$ is a right $\mathbb{S}_n$-module, where $\mathbb{S}_n$ is the symmetric group,  and this action is compatible, in a suitable way, with the composition of operations. 

An ideal $\mathcal{I} $ of $\mathcal{P} $ is a $\mathbb{S} $-submodule $(\mathcal{I}(n))_{n\in \mathbf{N}} $ of $\mathcal{P} $ such that, for all $p\in \mathcal{P}(k), p_i\in \mathcal{P}(n_i)$, if $p $ or one of the $p_i $ lies in $\mathcal{I}$, then $p\circ (p_1,\ldots,p_n)\in \mathcal{I} $. Given an ideal $\mathcal{I} $ of $\mathcal{P} $, we can define the quotient operad $\mathcal{P}/\mathcal{I}= \left(\frac{\mathcal{P}(n)}{\mathcal{I}(n)}\right)_{n\in \mathbf{N}} $. 

Let $\mathcal{P} $ be an operad. A $\mathcal{P} $-algebra is a vector space $V$ endowed with linear maps $$\mathcal{P}(n)\otimes V^{\otimes n}\to V ,$$ for all $n\in \mathbf{N}$, that are $\mathbb{S}_n $-equivariant and compatible with the composition. The module $\mathcal{P}(n) $ is the vector space of all possible expressions in $n$ variables for a given type of algebra and those maps are the evaluation of such an expression in $n$ elements of $V$. 

The forgetful functor from the category of operads into the category of $\mathbb{S} $-modules admits a left adjoint functor $\mathcal{P} $. For an $\mathbb{S}$-module $E$, we denote by $\mathcal{P}_E $ the free operad over $E$. It can be described in terms of trees. An element of $\mathcal{P}_E $ is a tree decorated by elements of $E$ such that a vertex $v$ of arity $n $ is decorated by an element of $E(n)$. The tree 
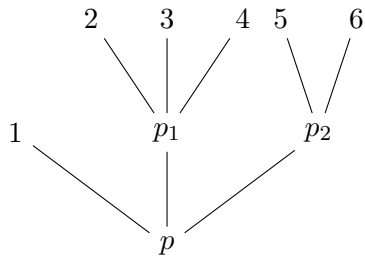
\[\begin{tikzpicture}[
  grow=up,  baseline=(current bounding box.center)
  level distance=1cm,
  sibling distance=2cm
]
\node {$p$}
  child { node {$p_2$} [sibling distance=1cm]
    child { node {$6$} }
    child { node {$5$} }
  }
  child { node {$p_1$} [sibling distance=1cm]
   child { node {$4$} }
   child {node {$3$} }
   child {node {$2$} } }
  child{ node {$1$}};
\end{tikzpicture}\] represents the composition $p\circ(I,p_1,p_2) $ with $p,p_1\in E(3)$, $p_2\in E(2)$ and $p\circ(I,p_1,p_2)\in E(6)$.

\subsubsection{The Lie triple systems operad}

We define here the operad of Lie triple systems. We consider the three operations $$X_1= [a_2,a_3,a_1], X_2= [a_3,a_1,a_2], X_3= [a_1,a_2,a_3], $$ where $X_1$ represents the operation $(a_1,a_2,a_3)\mapsto [a_2,a_3,a_1] $, etc. Let $E$ be the $\mathbf{K}$-vector space  generated by $X_1,X_2,X_3 $, on which  we define a right action of $\mathcal{S}_3$ by $X_i^\sigma=\varepsilon(\sigma) X_{\sigma^{-1}(i)}. $ This action is compatible with the skew-symmetry of the ternary product of a Lie triple system. The $\mathbb{S}$-module $\mathbf{E}$ is defined by $\mathbf{E}(n)= 0$ for $n\neq 3$ and $\mathbf{E}(3)=E$. We consider the free operad $\mathcal{P}_E$ generated by $\mathbf{E}$. It consists of all trees decorated by elements of $E$, and the action of $\mathbb{S}_n$ on $\mathcal{P}_E(n)$ is given by the permutation of the entries of the tree. It is compatible with the action on $E$; in fact, we have \[\begin{tikzpicture}[
  grow=up,  baseline=(current bounding box.center)
  level distance=1cm,
  sibling distance=1cm
]
\node {$X_i^{\sigma}$}
  child { node {$3$} [sibling distance=1cm]}
  child { node {$2$} [sibling distance=1cm] }
  child{ node {$1$}};
\end{tikzpicture} = \begin{tikzpicture}[
  grow=up,  baseline=(current bounding box.center)
  level distance=1cm,
  sibling distance=1.5cm
]
\node {$X_i$}
  child { node {$\sigma^{-1}(3)$} [sibling distance=1cm]}
  child { node {$\sigma^{-1}( 2)$} [sibling distance=1cm] }
  child{ node {$\sigma^{-1}( 1)$}};
\end{tikzpicture}. \] 

In $\mathcal{P}_E(5)$, we consider $A=X_3\circ(X_3,I,I) $ and $B= X_3\circ (I,I,X_3)$. The relations in the definition of a Lie triple system correspond to the relations $$X_1+X_2+X_3=0 $$ and $$B-A+ A^{(1\ 2\ 3)}- B^{(1\ 3)(2\ 4)}=0. $$ The operad $\mathbf{LTS}$ is $\mathcal{P}_E/R $, where $R$ is the ideal of $\mathcal{P}_E$ generated by $X_1+X_2+X_3 $ and $B-A+ A^{(1\ 2\ 3)}- B^{(1\ 3)(2\ 4)} $.

\begin{prop}
    The $\mathbf{LTS}-$algebras are the Lie triple systems.
\end{prop}

\subsubsection{Splitting of the Lie triple system operad}

Here we use the procedure described in \cite{PBG13} to split the operad $\mathbf{LTS} $. For each integer $n\geq 1$, denote $[n]=\left\{1,2,\cdots,n \right\}. $

\begin{defi}[\cite{PBG13}]
    \begin{enumerate}
        \item Let $\Omega $ be a set. A decorated tree $t$ is a tree together with  decorations on the vertices  by elements of $\Omega$ and  decorations on the leaves  by distinct positive integers. We denote by $t(\Omega) $ the set of decorated trees of $t$ and $\mathcal{T}(\Omega) $ the set of all decorated trees.
        \item For $t\in \mathcal{T}(\Omega)$, $Vin(t) $ (resp. $Lin(t)$) denote the set (resp. ordered set) of labels of the vertices (resp. leaves) of $t$.
    \end{enumerate}
\end{defi}

\begin{defi}[\cite{PBG13}]
    \begin{enumerate}
        \item For any $1\leq m\leq n $, let $N_{(n,m)} $ denote the set of all nonempty subsets of $[n] $ with at most $m$ elements. In particular, $$N_{(n,1)}:= A_n:=\left\{\left\{1 \right\}, \left\{2 \right\},\cdots, \left\{n \right\} \right\}, N_{(n,n)}:= B_n:= \left\{J\subseteq [n] \mid J\neq \varnothing \right\}. $$
        \item Let $t$ be a decorated tree and $\varnothing \neq J\subseteq L{in}(t). $ For $w\in V{in}(t) $, let $t_w$ denote the sub-tree of $t$ with root $w$. Denote $$J \sqcap w:=J\sqcap (w,t):=\left\{ i\in In(w)\mid J\cap Lin(t_i)\neq \varnothing \right\}\subseteq [l] ,$$ where $l$ is the number of inputs of $w$.
    \end{enumerate}
\end{defi}

\begin{ex}\cite{PBG13}
    Consider the case when $Lin(t)= [3] $. Then,  $w\in Vin(t) $  has arity $2$ or $3$. If $\left | Vin(w)\right |=3  $, then $t_w=t$ and $J\sqcap w= J $ for each $J\subseteq [3] $. If $ \left | Vin(w)\right |=2 $, then $w$ can appear  in four locations in $t$, denoted $w_i$, $1\leq i\leq 4$: $$t= w_1(1\vee w_2(2\vee 3)), t= w_3(w_4(1\vee 2)\vee 3). $$ For $J=2$, we have $$J\sqcap w_1= \left\{2\right\}, J\sqcap w_2= \left\{1\right\}, J\sqcap w_3= \left\{1\right\}, J\sqcap w_4= \left\{2\right\}. $$ For $J= \left\{1,3 \right\}$, we have $$J\sqcap w_1= \left\{1,2\right\}, J\sqcap w_2= \left\{2\right\}, J\sqcap w_3= \left\{1,2\right\}, J\sqcap w_4= \left\{1\right\}. $$
\end{ex}

\begin{defi}[\cite{PBG13}]
    \begin{enumerate}
        \item A configuration is a sequence $\mathcal{C}=(C_n)_{n\geq 1}$ with $C_n\subseteq B_n$ such that for any $J\in C_n$, decorated $n$-tree $t$ and $w\in Vin(t)$, we have $J\sqcap w\in C_{\left | In(w)\right |} $ whenever $J\sqcap w\neq \varnothing $.
        \item A configuration $\mathcal{C}=(C_n)$ is called $\mathcal{S}-$invariant if $C_n^{\mathcal{S}_n}\subseteq C_n, n\geq 1 $.
    \end{enumerate}
\end{defi}

\begin{ex}[\cite{PBG13}]
  \begin{enumerate}
      \item  $\mathcal{A}=(A_n) $ is a configuration.
      \item $\mathcal{B}=(B_n) $ is a configuration.
  \end{enumerate} 
\end{ex}

\begin{defi}[\cite{PBG13}]
    Let $V=\bigoplus\limits_{n\geq 1}V_n $ be a graded vector space with basis $\mathcal{V}= \bigsqcup\limits_{n\geq 1}\mathcal{V}_n $ and $\mathcal{C}$ be a configuration. \begin{enumerate}
        \item Define a graded vector space $\mathcal{C}Sp(V) $ by $$\mathcal{C}Sp(V)_n= V_n\otimes (\bigoplus\limits_{I\in C_n}\mathbf{K}e_I), $$ where we denote $(w\otimes e_I) $ by $\binom{w}{e_I} $ for $w\in \mathcal{V}_n$. Then $\left\{\binom{w}{e_I}\mid w\in \mathcal{V}_n,n\geq 1, I\in C_n \right\} $ is a basis of $\mathcal{C}Sp(V)_n $.
        \item For a labeled $n$-tree $t$ decorated by elements in $\mathcal{V} $, define $\mathcal{C}Sp(t) $ by \begin{itemize}
            \item $\mathcal{C}Sp(|)= | $,
            \item when $n\geq 2$, $\mathcal{C}Sp(t) $ is obtained by replacing each decoration $w\in Vin(t)\cap \mathcal{V}_l $ by $$\tbinom{w}{\star}:=\tbinom{w}{\star_\mathcal{C}}= \sum\limits_{I\in C_l}\tbinom{w}{e_I}. $$
        \end{itemize} We extend this definition to labeled tree decorated by elements in $V$ by linearity.
    \end{enumerate}
\end{defi}

\begin{defi}[\cite{PBG13}]
    Let $V=\bigoplus\limits_{n\geq 1}V_n $ be a graded vector space with basis $\mathcal{V}=\bigsqcup\limits_{n\geq 1} \mathcal{V}_n $ and let $\mathcal{C}$ be a configuration. Let $t $ be a labeled $n$-tree decorated by elements in $\mathcal{V} $ and let $J\in C_{\left | Lin(t)\right |} $. The splitting with configuration $\mathcal{C} $ $\mathcal{C}Sp_J(t) $ of $t$ with respect to $J$ is an element of the free non-symmetric operad $\mathcal{P}_{ns}(\mathcal{C}Sp(V)) $ over $\mathcal{C}Sp(V) $, defined by induction on $n:= \left | Lin(t)\right |$ as follows: \begin{itemize}
        \item $\mathcal{C}Sp_J(|)=|; $
        \item assume that $\mathcal{C}Sp_J(t) $ have been defined for $t$ with $ \left | Lin(t)\right |\leq k $ for a $k\geq 1$. Then, for a labeled $(k+1) $-tree $t$ with its decomposition $t=w(t_1\vee t_2\vee \cdots\vee t_l) $, $w\in \mathcal{V}_l$, denote $I:= J\sqcap t\in C_l $ and define $$\mathcal{C}Sp_J(t):= \tbinom{w}{e_I}\left(\bigvee\limits_{i=1}^l\mathcal{C}Sp_{J\cap Lin(t_i)}(t_i)\right)=\tbinom{w}{e_I}( \mathcal{C}Sp_{J\cap Lin(t_1)}(t_1)\vee \cdots \vee \mathcal{C}Sp_{J\cap Lin(t_l)}(t_l)), $$ using the notation $\mathcal{C}Sp_{\varnothing}(t)=\mathcal{C}Sp(t) $ from the previous definition.
    \end{itemize} 
\end{defi}

\begin{prop}[\cite{PBG13}]
    Let $V=\bigoplus\limits_{n\geq 1}V_n $ be a graded vector space with basis $\mathcal{V}=\bigsqcup\limits_{n\geq 1}\mathcal{V}_n $ and $\mathcal{C}$ be a configuration. Let $t$ be a labeled tree decorated by elements in $\mathcal{V}$ and $J\in C_{\left | Lin(t)\right |} $. The $\mathcal{C}-$splitting $\mathcal{C}Sp_J(t) $ is obtained by relabeling each vertex $w\in \mathcal{V}_l$ of $t$ by $\binom{w}{e_I} $ if $I:= J\sqcap w\neq \varnothing $ and by $\binom{w}{\star_\mathcal{C}}:= \sum\limits_{I\in C_l}\binom{w}{e_I} $ if $J\sqcap w=\varnothing $.
\end{prop}

\begin{ex}[\cite{PBG13}]
    For the configuration $\mathcal{A}$, we have \[
\mathcal{C}Sp_{\left\{{x_2}\right\}}\left(\begin{tikzpicture}[
  grow=up,  baseline=(current bounding box.center)
  level distance=1cm,
  sibling distance=2cm
]
\node {$w_3$}
  child { node {$w_2$} [sibling distance=1cm]
    child { node {$x_6$} }
    child { node {$x_5$} }
  }
  child { node {$w_1$} [sibling distance=1cm]
   child { node {$x_4$} }
   child {node {$x_3$} }
   child {node {$x_2$} } }
  child{ node {$x_1$}};
\end{tikzpicture}\right)= \begin{tikzpicture}[
  grow=up,  baseline=(current bounding box.center)
  level distance=1cm,
  sibling distance=2cm
]
\node {$\binom{w_3}{e_2}$}
  child { node {$\binom{w_2}{\star_\mathcal{A}}$} [sibling distance=1cm]
    child { node {$x_6$} }
    child { node {$x_5$} }
  }
  child { node {$\binom{w_1}{e_1}$} [sibling distance=1cm]
   child { node {$x_4$} }
   child {node {$x_3$} }
   child {node {$x_2$} } }
  child{ node {$x_1$}};
\end{tikzpicture}\] \[= \begin{tikzpicture}[
  grow=up,  baseline=(current bounding box.center)
  level distance=1cm,
  sibling distance=2cm
]
\node {$\binom{w_3}{e_2}$}
  child { node {$\binom{w_2}{e_1}$} [sibling distance=1cm]
    child { node {$x_6$} }
    child { node {$x_5$} }
  }
  child { node {$\binom{w_1}{e_1}$} [sibling distance=1cm]
   child { node {$x_4$} }
   child {node {$x_3$} }
   child {node {$x_2$} } }
  child{ node {$x_1$}};
\end{tikzpicture} + \begin{tikzpicture}[
  grow=up,  baseline=(current bounding box.center)
  level distance=1cm,
  sibling distance=2cm
]
\node {$\binom{w_3}{e_2}$}
  child { node {$\binom{w_2}{e_2}$} [sibling distance=1cm]
    child { node {$x_6$} }
    child { node {$x_5$} }
  }
  child { node {$\binom{w_1}{e_1}$} [sibling distance=1cm]
   child { node {$x_4$} }
   child {node {$x_3$} }
   child {node {$x_2$} } }
  child{ node {$x_1$}};
\end{tikzpicture}
\]
\end{ex}

We describe now how to split an operad.

\begin{defi}[\cite{PBG13}]
    Let $\mathcal{P}=\mathcal{P}_V/(R) $ be an operad, where $V=\bigoplus\limits_{n\geq 1}V(n) $ is an $\mathbb{S} $-module with a linear basis $\mathcal{V}=\bigsqcup\limits_{n\geq 1}\mathcal{V}_n $  that is invariant under the action of $\mathbb{S}_n $ and where the space of relations $(R)$ is generated, as an $\mathbb{S}- $module, by a set $R$ of locally homogeneous elements $$r_s= \sum\limits_{i}c_{s,i}t_{s,i}, \ c_{s,i}\in \mathbf{K}, \ t_{s,i}\in \mathcal{P}_\mathcal{V}, \ 1\leq s\leq k. $$ Let $\mathcal{C} $ be an $\mathbb{S} $-invariant configuration. The $\mathcal{C}$-splitting of $\mathcal{P} $ is defined to be the operad $$\mathcal{C}Sp(\mathcal{P})= \mathcal{P}_{\mathcal{C}Sp(V)}/(\mathcal{C}Sp(R)), $$ where the $\mathbb{S}_n $-action on $\mathcal{C}Sp(V)(n)=V(n)\otimes (\bigoplus\limits_{I\in C_n}\mathbf{K}e_I) $ is given by $$\tbinom{w}{e_I}^\sigma:= \tbinom{w^{\sigma}}{e_{\sigma(I)}}, \ w\in V(n), \ \sigma(I)= \left\{\sigma(i)\mid i\in I \right\}  $$ and the space of relations is generated, as an $\mathbb{S}- $module, by $$\mathcal{C}Sp(R):= \left\{\mathcal{C}Sp_J(r_s)= \sum\limits_{i}c_{s,i}\mathcal{C}Sp_J(t_{s,i}) \mid J\in \mathcal{C}_{\left | Lin(t_{s,i})\right |}, \ 1\leq s\leq k \right\}. $$
\end{defi}

\begin{ex}
    \begin{enumerate}
    \item $\mathcal{B}Sp(\mathbf{Lie})= \mathbf{post}-\mathbf{Lie} $.
        \item $\mathcal{A}Sp(\mathbf{Lie})= \mathbf{pre}-\mathbf{Lie}. $ We compute this last exemple. The Lie operad is equal to $\mathcal{P}_E/(R) $, where $E$ is the $\mathbb{S} $-module defined by $E(n)=0 $ for $n\neq 2$ and $E(2)$ is the signature $\mathbb{S}_2 $-module generated by $[\cdot,\cdot] $ which we will denote by $w$ and $(R)$ is the $\mathbb{S}_3 $-submodule of $\mathcal{P}_E$ generated by $R=w\circ (w,I) + w\circ (w,I)^{(1 \ 2 \ 3)}+ w\circ (w,I)^{(1 \ 3 \ 2)} $. So $\mathcal{A}Sp(\mathbf{Lie})=\mathcal{P}_{\mathcal{A}Sp(E)}/(\mathcal{A}Sp(R)) $. $\mathcal{A}Sp(E)(n)=0 $ for $n\neq 2$ and $\mathcal{A}Sp(E)(2)$ is the $\mathbb{S}_2 $-module generated by $\binom{w}{e_{\left\{1\right\}}} $ because $\binom{w}{e_{\left\{2\right\}}}^{(1 \ 2)}= \binom{w^{(1 \ 2)}}{e_{\left\{1\right\}}}= -\binom{w}{e_{\left\{1\right\}}}  $. \[R= \begin{tikzpicture}[
  grow=up,  baseline=(current bounding box.center)
  level distance=1cm,
  sibling distance=2cm
]
\node {$w$}
  child { node {$I$} [sibling distance=1cm]
    child { node {$x_3$} }
  }
  child { node {$w$} [sibling distance=1cm]
   child {node {$x_2$} }
   child {node {$x_1$} } };
\end{tikzpicture}+ \begin{tikzpicture}[
  grow=up,  baseline=(current bounding box.center)
  level distance=1cm,
  sibling distance=2cm
]
\node {$w$}
  child { node {$I$} [sibling distance=1cm]
    child { node {$x_2$} }
  }
  child { node {$w$} [sibling distance=1cm]
   child {node {$x_1$} }
   child {node {$x_3$} } };
\end{tikzpicture} + \begin{tikzpicture}[
  grow=up,  baseline=(current bounding box.center)
  level distance=1cm,
  sibling distance=2cm
]
\node {$w$}
  child { node {$I$} [sibling distance=1cm]
    child { node {$x_1$} }
  }
  child { node {$w$} [sibling distance=1cm]
   child {node {$x_3$} }
   child {node {$x_2$} } };
\end{tikzpicture}. \] $$\mathcal{A}Sp_{\left\{1\right\}}(R)= \tbinom{w}{e_1}\left(\tbinom{w}{e_1}(x_1,x_2),x_3\right)+ \tbinom{w}{e_1}\left(\tbinom{w}{e_2}(x_3,x_1),x_2\right)+ \tbinom{w}{e_2}\left(\tbinom{w}{\star_{\mathcal{A}}}(x_2,x_3),x_1\right), $$ with $\binom{w}{e_2}= -\binom{w}{e_1}^{(1 \ 2)} $ and $\binom{w}{\star_\mathcal{A}}= \binom{w}{e_1}+\binom{w}{e_2}= \binom{w}{e_1}-\binom{w}{e_1}^{(1 \ 2)} $ so \begin{align*}
    \mathcal{A}Sp_{\left\{1\right\}}(R)=&\tbinom{w}{e_1}(\tbinom{w}{e_1}(x_1,x_2),x_3)- \tbinom{w}{e_1}(\tbinom{w}{e_1}(x_1,x_3),x_2)- \tbinom{w}{e_1}(x_1,\tbinom{w}{\star_{\mathcal{A}}}(x_2,x_3)) \\ =& \tbinom{w}{e_1}(\tbinom{w}{e_1}(x_1,x_2),x_3)- \tbinom{w}{e_1}(\tbinom{w}{e_1}(x_1,x_3),x_2)- \tbinom{w}{e_1}(x_1,\tbinom{w}{e_1}(x_2,x_3)) \\ &+\tbinom{w}{e_1}(x_1,\tbinom{w}{e_1}(x_3,x_2)) . 
\end{align*} If we denote $\circ $ the operation $\binom{w}{e_1} $, $\mathcal{A}Sp_{\left\{1\right\}}(R) $ is exactly the same as the identity $$(x_1\circ x_2)\circ x_3 - (x_1\circ x_3)\circ x_2 - x_1\circ(x_2\circ x_3) + x_1\circ (x_3\circ x_2) $$ which defines a right-symmetric algebra. Hence, $\mathcal{A}Sp_{\left\{2\right\}}(R) $ and $\mathcal{A}Sp_{\left\{3\right\}}(R) $ give the same relation. 
    \end{enumerate}
\end{ex}

We now  split the operad $\mathbf{LTS} $. We have $\mathcal{A}Sp(E)(n)=0 $ for $n\neq 3$ and $\mathcal{A}Sp(E)(3) $ is the $\mathbb{S}_3$-module generated by $\binom{X_1}{e_1} $ and $\binom{X_3}{e_1} $. Indeed, we have the relations $\binom{X_i}{e_j}^\sigma= \varepsilon(\sigma)\binom{X_{\sigma^{-1}(i)}}{e_{\sigma(j)}} $ for all $\sigma\in \mathbb{S}_3$. We write the identities of a Lie triple system in terms of trees. $ X_1+X_2+X_3$ is equivalent to \[r_1= \begin{tikzpicture}[
  grow=up,  baseline=(current bounding box.center)
  level distance=1cm,
  sibling distance=1cm
]
\node {$X_1$}
  child { node {$x_3$} [sibling distance=1cm] }
  child { node {$x_2$} [sibling distance=1cm] }
  child { node {$x_1$} [sibling distance=1cm] };
\end{tikzpicture}+ \begin{tikzpicture}[
  grow=up,  baseline=(current bounding box.center)
  level distance=1cm,
  sibling distance=1cm
]
\node {$X_2$}
  child { node {$x_3$} [sibling distance=1cm] }
  child { node {$x_2$} [sibling distance=1cm] }
  child { node {$x_1$} [sibling distance=1cm] };
\end{tikzpicture} + \begin{tikzpicture}[
  grow=up,  baseline=(current bounding box.center)
  level distance=1cm,
  sibling distance=1cm
]
\node {$X_3$}
  child { node {$x_3$} [sibling distance=1cm] }
  child { node {$x_2$} [sibling distance=1cm] }
  child { node {$x_1$} [sibling distance=1cm] };
\end{tikzpicture}. \] $\mathcal{A}Sp_{\left\{1\right\}}(r_1)= \binom{X_1}{e_1}+ \binom{X_2}{e_1}+\binom{X_3}{e_1}=\binom{X_1}{e_1}- \binom{X_3}{e_1}^{(2 \ 3)}+\binom{X_3}{e_1}  $ and $\mathcal{A}Sp_{\left\{2\right\}}(r_1),\mathcal{A}Sp_{\left\{3\right\}}(r_1) $ give the same relation. $B-A+ A^{(1\ 2\ 3)}- B^{(1\ 3)(2\ 4)}$ is equivalent to \[r_2= \begin{tikzpicture}[
  grow=up,  baseline=(current bounding box.center)
  level distance=1cm,
  sibling distance=1cm
]
\node {$X_3$}
  child { node {$X_3$} [sibling distance=1cm]
    child { node {$x_5$} }
    child {node {$x_4$} }
    child { node {$x_3$} }
  }
  child { node {$x_2$} [sibling distance=1cm] }
  child {node {$x_1$} } ;
\end{tikzpicture}- \begin{tikzpicture}[
  grow=up,  baseline=(current bounding box.center)
  level distance=1cm,
  sibling distance=1cm
]
\node {$X_3$}
  child { node {$x_5$} [sibling distance=1cm] }
  child {node {$x_4$} } 
  child { node {$X_3$} [sibling distance=1cm]
    child { node {$x_3$} }
    child {node {$x_2$} }
    child { node {$x_1$} }
  };
\end{tikzpicture} - \begin{tikzpicture}[
  grow=up,  baseline=(current bounding box.center)
  level distance=1cm,
  sibling distance=1cm
]
\node {$X_3$}
  child { node {$x_5$} [sibling distance=1cm] }
   child { node {$X_3$} [sibling distance=1cm]
    child { node {$x_4$} }
    child {node {$x_2$} }
    child { node {$x_1$} }}
  child {node {$x_3$} } ;
\end{tikzpicture} - \begin{tikzpicture}[
  grow=up,  baseline=(current bounding box.center)
  level distance=1cm,
  sibling distance=1cm
]
\node {$X_3$}
  child { node {$X_3$} [sibling distance=1cm]
    child { node {$x_5$} }
    child {node {$x_2$} }
    child { node {$x_1$} }
  }
  child { node {$x_4$} [sibling distance=1cm] }
  child {node {$x_3$} } ;
\end{tikzpicture} . \] Now, we deal with $\mathcal{P}_{\mathcal{A}Sp(E)}/(\mathcal{A}Sp(r_1)) $ so $$\tbinom{X_3}{\star_\mathcal{A}}= \tbinom{X_3}{e_1}+\tbinom{X_3}{e_2}+\tbinom{X_3}{e_3}= \tbinom{X_3}{e_1}- \tbinom{X_3}{e_1}^{(1 \ 2)}+ \tbinom{X_3}{e_1}^{(1 \ 3)}- \left(\tbinom{X_3}{e_1}^{(2 \ 3)}\right)^{(1 \ 3)} .$$ \begin{align*}
    \mathcal{A}Sp_{\left\{1\right\}}(r_2) =&\tbinom{X_3}{e_1}\left(x_1,x_2,\tbinom{X_3}{\star_\mathcal{A}}(x_3,x_4,x_5) \right)- \tbinom{X_3}{e_1}\left(\tbinom{X_3}{e_1}(x_1,x_2,x_3),x_4,x_5 \right)\\ & - \tbinom{X_3}{e_2}\left(x_3,\tbinom{X_3}{e_1}(x_1,x_2,x_4),x_5 \right)- \tbinom{X_3}{e_3}\left(x_3,x_4,\tbinom{X_3}{e_1}(x_1,x_2,x_5) \right)  \\ =& \tbinom{X_3}{e_1}\left(x_1,x_2,\tbinom{X_3}{e_1}(x_3,x_4,x_5) \right) - \tbinom{X_3}{e_1}\left(x_1,x_2,\tbinom{X_3}{e_1}(x_4,x_3,x_5) \right)\\ & + \tbinom{X_3}{e_1}\left(x_1,x_2,\tbinom{X_3}{e_1}(x_5,x_4,x_3) \right)- \tbinom{X_3}{e_1}\left(x_1,x_2,\tbinom{X_3}{e_1}(x_5,x_3,x_4) \right) \\ &- \tbinom{X_3}{e_1}\left(\tbinom{X_3}{e_1}(x_1,x_2,x_3),x_4,x_5 \right) - \tbinom{X_3}{e_1}\left(\tbinom{X_3}{e_1}(x_1,x_2,x_4),x_3,x_5 \right) \\ &- \tbinom{X_3}{e_1}\left(\tbinom{X_3}{e_1}(x_1,x_2,x_5),x_4,x_3 \right) +  \tbinom{X_3}{e_1}\left(\tbinom{X_3}{e_1}(x_1,x_2,x_5),x_3,x_4 \right). 
\end{align*} Then,  $\mathcal{A}Sp_{\left\{1\right\}}(r_2) $ corresponds to Equation \eqref{equ7} in  \hyperref[d4.6]{Definition~\ref*{d4.6}}. Moreover, $\mathcal{A}Sp_{\left\{2\right\}}(r_2) $ gives the same identity. \begin{align*}
    \mathcal{A}Sp_{\left\{3\right\}}(r_2) =& \tbinom{X_3}{e_3}\left(x_1,x_2,\tbinom{X_3}{e_1}(x_3,x_4,x_5) \right)- \tbinom{X_3}{e_1}\left(\tbinom{X_3}{e_3}(x_1,x_2,x_3),x_4,x_5 \right)\\ & - \tbinom{X_3}{e_1}\left(x_3,\tbinom{X_3}{\star_\mathcal{A}}(x_1,x_2,x_4),x_5 \right) - \tbinom{X_3}{e_1}\left(x_3,x_4,\tbinom{X_3}{\star_\mathcal{A}}(x_1,x_2,x_5) \right).
\end{align*} 
Furthermore, $\mathcal{A}Sp_{\left\{3\right\}}(r_2) $ corresponds to Equation \eqref{equ8} in  \hyperref[d4.6]{Definition~\ref*{d4.6}}. Indeed, with the notation of  \hyperref[d4.6]{Definition~\ref*{d4.6}}, if we denote the operation $\tbinom{X_3}{e_1} $ by $\left\{\cdot,\cdot,\cdot\right\} $ then $\tbinom{X_3}{e_3}= \left\{\cdot,\cdot,\cdot\right\}^\wedge $ and $\tbinom{X_3}{\star_{\mathcal{A}}}= [\cdot,\cdot,\cdot]_C $. We denote by $\mathbf{pre-LTS} $ the operad $\mathcal{P}_{\mathcal{A}Sp(E)}/(\mathcal{A}Sp(r_1), \mathcal{A}Sp_{\left\{1\right\}}(r_2),     \mathcal{A}Sp_{\left\{3\right\}}(r_2)) $. This is the operad of pre-Lie triple systems.

\begin{prop}
    The $\mathbf{pre-LTS} $-algebras are the pre-Lie triple systems.
\end{prop}

\begin{proof}
    Let $T$ be a vector space. Given the definition of the operad $\mathbf{pre-LTS} $ with generators and relations, endowing $T$ with a structure of a $\mathbf{pre-LTS}$-algebra is equivalent to giving a map $\left\{\cdot,\cdot,\cdot\right\}: T\times T\times T\to T $ satisfying Equations \eqref{equ7} and \eqref{equ8}.
\end{proof}

\section{Rota-Baxter operators on restricted Lie triple systems}\label{s5}
The aim of this section is to introduce a definition of Rota-Baxter operators on restricted Lie triple systems and provide some properties. We discuss Jacobson's Theorem and relationships with representations.
\subsection{Definitions}

Let $T$ be a Lie triple system and $n\geq 3$ be any positive odd integer. For elements $x_1,x_2,\dots , x_n $ of $T$, define \begin{center}
    $(x_1,x_2,\dots,x_n)= [[\cdots [[x_1,x_2,x_3],x_4,x_5],\cdots],x_{n-1},x_n] \in T$.
\end{center} For $x,y\in T $, the expression $(x,Zx+y,Zx+y,\dots , Zx+y) $,  with $Zx+y$ occurring $p-1$ times, is a polynomial in $Z$ with coefficient in $T$. For $i=1,\dots,p $, we define $s_i(x,y) \in T$ by requiring $is_i(x,y)$ to be the coefficient of $Z^{i-1} $ in $(x,Zx+y,Zx+y,\dots,Zx+y) $. For any injective embedding $L$ of $T$, we have \begin{center}
    $(x,Zx+y,Zx+y,\cdots,Zx+y)= [Zx+y,[Zx+y,[\cdots ,[Zx+y,x]\cdots]]] = (\ad(Zx+y))^{p-1}(x) $.
\end{center} In this case, $is_i(x,y) $ is the coefficient of $Z^{i-1} $ in the expression $((\ad(Zx+y))^{p-1}(x) $, which agrees with the use of the notation $s_i(x,y) $ in the definition of a restricted Lie algebra. 

\begin{defi}\label{D5.1}
    A Lie triple system $T$ over a field of characteristic $p>2$ is said restricted if there is given a $p$-map $(\cdot)^{[p]}:T\to T $  such that the following conditions are satisfied : \begin{enumerate}        
        \item $(\alpha x)^{[p]}= \alpha^p x^{[p]} $ for all $\alpha$ in $\mathbf{K}$ and all $x$ in $T$
        \item $(x+y)^{[p]}= x^{[p]}+y^{[p]}+ \sum\limits_{i=1}^{p-1}s_i(x,y) $ for all $x,y$ in $T$,
        \item  $[x,y^{[p]},z]=(x,y,\cdots,y,z) $($p$ copies of y) for all $x,y,z$ in $T$,
    \end{enumerate} where $is_i(x,y)$ is the coefficient of $Z^{i-1} $ in $(x,Zx+y,Zx+y,\dots,Zx+y) $, and $Zx+y $ appears $p-1$ times.
\end{defi}

\begin{defi}
    Let $(T,(\cdot)^{[p]}), (T',(\cdot)^{[p]'})$ be two restricted Lie triple systems. A morphism of restricted Lie triple systems between $T$ and $T'$ is a morphism of Lie triple systems $f:T\to T' $ such that $f(x^{[p]})=f(x)^{[p]'} $ for all $x\in T$.
\end{defi}

\begin{ex}
    Let $L$ be a restricted Lie algebra. Then $L$ equipped with the bracket $[x,y,z]=[[x,y],z] $ is a restricted Lie triple system. 
\end{ex}

\begin{rmq}
    As for restricted Lie algebra, a $p$-map of a restricted Lie triple system is determined by its values on a given basis. Moreover, if $T$ is centerless, then the $p$-map, if it exists, is unique.
\end{rmq}

There is an analogue of Jacobson's theorem for Lie triple systems.

\begin{thm}[\cite{HP02}]\label{thm5.3}
    Let $T$ be a Lie triple system. Suppose there exists a basis $(u_{i})_{i\in J}$  of $T$ with the property that for every $i$ there exists a $v_{i}\in T_{p} $ such that \begin{center}
        $[a,v_{i},c]= (a,u_{i},\dots,u_{i},c) $ for all $a,c\in T$
    \end{center} Then, there is a unique restricted structure on $T$ such that $u_{i}^{[p]}=v_{i} $ for all $i$.
    
\end{thm}

\begin{defi}
    Let $(T,(\cdot)^{[p]})$ and $ (T',(\cdot)^{[p]'}) $ be two restricted Lie triple systems. A map $\varphi: T\to T'$ is a morphism of restricted Lie triple systems if it is a morphism of Lie triple systems such that $\varphi(x^{[p]})= \varphi(x)^{[p]'}$ for all $x\in T$.
\end{defi}

\begin{defi}[\cite{BM26}]
    Let $T$ be a restricted Lie triple system, and $(V,\theta)$ be a representation of $T$. The pair $(V,\theta)$ is called a restricted representation if in addition we have that, for all $x,y\in T  $,  
    \begin{equation*}
    \theta(x^{[p]},y)=\theta(x,y) \circ \theta(x,x)^{\frac{p-1}{2}} 
    \end{equation*}  and 
    \begin{equation*} \theta(x,y^{[p]})= \theta(y,y)^{\frac{p-1}{2}}\circ\theta(x,y) .   \end{equation*} 
\end{defi}

\begin{ex}
    The pair $(\theta,T)$, where $\theta(x,y)(z)=[z,x,y] $, is a restricted representation, called the adjoint representation.
\end{ex}

The following  theorem was proved in \cite{BM26} in the colored case using an equivalent version of  \hyperref[thm5.3]{Theorem~\ref*{thm5.3}} in the colored case which is true whenever the $PBW$-Theorem is true in the colored case, i.e. when $p>3$. Since in the non-colored case, $PBW$-Theorem is true in characteristic $3$, the same proof as in \cite{BM26} shows the result in the non-colored case for $p=3$.

\begin{prop}\label{pro5.7}\cite{BM26}
    Let $T$ be a Lie triple system and $(V,\theta) $ be a restricted representation of $T$. Let $(x,v),(y,u),(z,w)\in T\oplus V$. Then we have \begin{center}
        $((y,u),(x,v),\dots,(x,v),(z,w))= ((y,x,\dots,x,z), \theta(x,z)\theta(x,x)^n(u) - \theta(x,z)\sum\limits_{k=1}^n\binom{2n+1}{2k}\theta(x,x)^{k-1}\theta(y,x)\theta(x,x)^{n-k}(v) -\theta(y,z)\theta(x,x)^n(v) + \theta(x,z)\sum\limits_{k=0}^{n-1}\binom{2n+1}{2k+2}\theta(x,x)^{n-k-1}\theta(x,y)\theta(x,x)^k(v) +\sum\limits_{k=0}^n \binom{2n+1}{2k+1}\theta(x,x)^{n-k}\theta(x,y)\theta(x,x)^k(w)-\sum\limits_{k=0}^n \binom{2n+1}{2k+1}\theta(x,x)^{k}\theta(x,y)\theta(x,x)^{n-k}(w) ) $,
    \end{center} where $(x,v)$ appears $2n+1$ times. 
In  particular, \begin{center}
        $((y,u),(x,v),\dots,(x,v),(z,w)) = ((y,x,\dots,x,z), \theta(x,z)\theta(x,x)^{\frac{p-1}{2}}(u) -\theta(y,z)\theta(x,x)^{\frac{p-1}{2}}(v) + \theta(x,y)\theta(x,x)^{\frac{p-1}{2}}(w) - \theta(x,x)^{\frac{p-1}{2}}\theta(y,x)(w))  $,
    \end{center} where $(x,v)$ occurs $p$ times.
\end{prop}

\begin{thm}\label{thm5.7}
     Let $T$ be a restricted Lie triple system over a field $\mathbf{K} $ of characteristic $p>2 $. Let $(V,\theta)$ be a restricted representation of $T$. Assume the operation $[\cdot,\cdot,\cdot]_V: (T\oplus V)\times (T\oplus V)\times (T\oplus V)\to T\oplus V  $ is defined  by \begin{center}
        $[(x,a),(y,b),(z,c)]_V=([x,y,z], \theta(y,z)(a)-\theta(x,z)(b)+ D(x,y)(c)) .$ 
    \end{center}  Then there exists a $p$-map on $T\oplus V$ given by the formula \begin{center}
        $(x_i,v_i)^{[p]}=(x_i^{[p]},\theta(x_i,x_i)^{\frac{p-1}{2}}(v_i))$
    \end{center} for any basis $(x_i,v_i) $ of $T\oplus V $.
\end{thm}

\begin{proof}
     Using \hyperref[pro5.7]{Proposition ~\ref*{pro5.7}}, we see that we can use Jacobson's Theorem. Indeed, we have \begin{align*}
        [(y&,u),(x^{[p]},\theta(x,x)^{\frac{p-1}{2}}(v)),(z,w)] \\ &= ([y,x^{[p]},z], \theta(x^{[p]},z)(u) - \theta(y,z)\theta(x,x)^{\frac{p-1}{2}}(v)\\& \quad\quad\quad\quad\quad\quad\quad\quad\quad + \theta(x^{[p]},y)(w) - \theta(y,x^{[p]})(w)) \\ & = ([y,x^{[p]},z], \theta(x,z)\theta(x,x)^{\frac{p-1}{2}}(u) - \theta(y,z)\theta(x,x)^{\frac{p-1}{2}}(v) \\ &\quad\quad\quad\quad\quad\quad\quad\quad\quad + \theta(x,y)\theta(x,x)^{\frac{p-1}{2}}(w) - \theta(x,x)^{\frac{p-1}{2}}\theta(y,x)(w)) \\ & = ((y,u),(x,v),\dots,(x,v),(z,w)).
    \end{align*} So, for any basis $((x_i,v_i))_{i\in I} $ of $T $, there exists a $p$-map $(\cdot)^{[p]}$ on $T\oplus V $ such that $(x_i,v_i)^{[p]}= (x_i^{[p]},\theta(x_i,x_i)^{\frac{p-1}{2}}(v_i)) $.
\end{proof}

\begin{rmq}
    The proof of this theorem uses Jacobson's Theorem. In fact, we prove that for all $(x,u),(y,v)$ and $(z,w)$ in $ T\oplus V$, $$[(y,u),(x^{[p]},\theta(x,x)^{\frac{p-1}{2}}(v)),(z,w)]= ((y,u),(x,v),\dots,(x,v),(z,w)). $$ Thus, each choice of a basis of $T\oplus V $ gives rise to a $p$-map by Jacobson's Theorem. If $T\oplus V$ is centerless then all these $p$-map coincide and the map $(x,y)\mapsto (x^{[p]},\theta(x,x)^{\frac{p-1}{2}}(y)) $ is the unique $p$-map on $T\oplus V$.
\end{rmq}

In the case of the adjoint representation of $T$, the $p$-map on $T\oplus T $ is given, on a basis, by $(x_i,y_i)^{[p]}= (x_i^{[p]}, (y_i,x_i,\ldots,x_i)) $. This leads to the following definition and  result.

\begin{defi}\label{d5.6}
    Let $T$ be a restricted Lie triple system and $P$ be a Rota-Baxter operator on $T$. We say that $P$ is  restricted if in addition the following identity holds: \begin{center}
        $P(x)^{[p]}= P((x,Px,\ldots, Px)) $ for all $x\in T$.
    \end{center} 
\end{defi}

\begin{ex}
    Let $(L,[\cdot,\cdot])$ be a restricted Lie algebra endowed with a restricted Rota-Baxter operator $P$. Then $P$ is a restricted Rota-Baxter operator of the induced restricted Lie triple system $(L,[\cdot,\cdot,\cdot]) $.
\end{ex}

\begin{prop}\label{p5.8}   
    Let $T$ be a restricted Lie triple system. A linear map $P:T\to T$ is a restricted Rota-Baxter operator if and only if $Gr(P) $ is a restricted Lie subsystem of $T\oplus T$, where $T\oplus T $ has a $p$-map defined in  \hyperref[thm5.7]{Theorem~\ref*{thm5.7}}, with respect to the adjoint representation for any basis of $T\oplus T $. 
\end{prop}

\begin{proof}
    It remains to show that $Gr(P) $ is closed under the $p$-map on $T\oplus T $ if and only if $P$ is restricted. First, we observe that for $Gr(P) $ to be closed under a $p$-map it is sufficient that for a basis $(Px_i,x_i) $ of $Gr(P)$, $(Px_i,x_i)^{[p]}\in Gr(P) $. The $p$-map defined on the semi-direct product $T\oplus T$ with respect to the adjoint representation is given, on a basis $(x_i,y_i) $, by $(x_i,y_i)^{[p]}=(x_i^{[p]},(y_i,x_i,\ldots,x_i)) $. We  choose a basis of $T\oplus T$ by taking a basis $(Px_i,x_i) $ of $Gr(P)$ and completing it to a basis of $T\oplus T$. Therefore, a Rota-Baxter operator must satisfy the condition $P(x_i)^{[p]}=P((x_i,Px_i,\ldots,Px_i))) $. In this case, $(Px_i,x_i)^{[p]}\in Gr(P)$ and $Gr(P) $ is a restricted subsystem of the restricted Lie triple system $T\oplus T$. This condition is exactly the definition of a restricted Rota-Baxter operator. Conversely, assume that $Gr(P) $ is stable under any $p$-map induced by a choice of basis of $T\oplus T$. Let $x\in T$ be a nonzero element. $(Px,x) $ can be completed to a basis of $T\oplus T $. Since $Gr(P)$ is stable under the corresponding $p$-map, we have $(Px,x)^{[^p]}=(P(x)^{[p]},(x,Px,\cdots,Px)\in Gr(P) $, which shows that $P $ is a restricted Rota-Baxter operator.
\end{proof}

\begin{rmq} If $T$ is centerless, then the semi-direct product $T\oplus T$ is also centerless. The map $(x,y)\mapsto (x^{[p]},(y,x,\ldots,x)) $ is then a $p$-map. 
   As in the case of restricted Rota-Baxter operators on restricted Lie algebras, the condition on the graph of $P$ shows that it suffices to verify that $P(e_i)^{[p]}=P((e_i,Pe_i,\ldots, Pe_i)) $ for a basis $(e_i)$ of $T$ to prove that $P$ is restricted.
\end{rmq}

\begin{ex}
   We consider  \hyperref[exa4.7]{Example~\ref*{exa4.7}}. Let $T$ be a $2$-dimensional Lie triple system with a basis $(e_1,e_2)$ and a bracket defined by $[e_1,e_2,e_2]=e_1, [e_2,e_1,e_2]=-e_1 $ and the other brackets are zero. By Jacobson's theorem on Lie triple systems, we see that we can define a $p$-map on $T$ by setting $$e_1^{[p]}=0, \ e_2^{[p]}=e_2. $$ 
   Notice that $T$ is centerless. Indeed, let $x=x_1e_1+x_2e_2 $ be in the center of $T$. We have that $[x,e_2,e_2]=0 $ so $x_1=0 $. We have that $[x,e_1,e_2]=0 $ so $x_2=0$ and $x=0$. Then, using the previous remark we see that the Rota-Baxter operator defined by $ P=\begin{pmatrix}
0 & a \\
0 & 0 \\
\end{pmatrix} $ is a restricted Rota-Baxter operator.
\end{ex}

\subsection{Restricted pre-Lie triple systems}

The purpose of this section is introduce the notion of restricted pre-Lie triple system, and to verify whether Jacobson's formula holds for pre-Lie triple systems in positive characteristic.

Dzhumadil'daev showed that Jacobson's identities hold in pre-Lie algebras. This   motivated  his definition of restricted pre-Lie algebras, which we call trivially restricted pre-Lie algebras. We will show that we also have Jacobson's identities in pre-Lie triple systems. Then we will define trivially restricted and restricted pre-Lie triple systems. For that, we first need  to establish some technical lemmas.

\begin{prop}\label{p5.6}\cite{BM26}
    Let $T$ be a Lie triple system and $\theta$ be a representation of $T$. Then for all $x,y\in T$ we have \begin{center}
        $\theta(x,[[[\dots[y,x,x],x,x]\dots],x,x])= \sum\limits_{k=0}^n\binom{2n}{2k}\theta(x,x)^{n-k}\theta(x,y)\theta(x,x)^k - \sum\limits_{k=0}^{n-1}\binom{2n}{2k+1}\theta(x,x)^k\theta(y,x)\theta(x,x)^{n-k} $,
    \end{center} and \begin{center}
        $\theta([[\dots[y,x,x]\dots],x,x],x)= \sum\limits_{k=0}^n\binom{2n}{2k}\theta(x,x)^k\theta(y,x)\theta(x,x)^{n-k} - \sum\limits_{k=0}^{n-1}\binom{2n}{2k+1}\theta(x,x)^{n-k}\theta(x,y)\theta(x,x)^k $,
    \end{center} where $x$ occurs $2n$ times in the brackets.
\end{prop}

\begin{prop}\label{p5.7}
    Let $(T,\left\{ \cdot,\cdot,\cdot\right\}) $ be a pre-Lie triple system and  $\theta(x,y)(z)= \left\{ z,x,y\right\} $ be the representation associated with the Lie triple system $(T,[\cdot,\cdot,\cdot]_C) $. Then for all $x,u\in T$, we have \begin{center}
        $(x,u,\ldots,u)_C= \theta(u,u)^n(x) + \sum\limits_{k=0}^{n-1}\binom{2n}{2k}\theta(u,u)^k\theta(u,x)\theta(u,u)^{n-1-k}(u) - \sum\limits_{k=0}^{n-1}\binom{2n}{2k+1}\theta(u,u)^{n-1-k}\theta(x,u)\theta(u,u)^k(u) $.
    \end{center} where $u$ occurs $2n$ times.
\end{prop}

\begin{proof}
    We perform by induction in $n$. The initialization is clear. Assume that the formula is true for $n\geq 1$. We have,  \begin{align*}
        [x,y,z]_C= &\left\{x,y,z\right\}-\left\{ y,x,z\right\} +  \left\{ z,y,x\right\} -\left\{ z,x,y\right\} \\  =& \theta(y,z)(x) - \theta(x,z)(y) + \theta(y,x)(z) - \theta(x,y)(z). 
    \end{align*} Thus we can compute $(x,u,\ldots,u)_C $ where $u$ occurs $2(n+1)$ times. Using the induction hypothesis and \hyperref[p5.6]{Proposition~\ref*{p5.6}}, we have \begin{align*}
        (x,u,\ldots,&u)_C = [(x,u,\ldots,u)_C,u,u]_C 
        \\ =& \theta(u,u)((x,u,\ldots,u)_C) - 2\theta((x,u,\ldots,u)_C,u)(u) + \theta(u,(x,u,\ldots,u)_C)(u) 
        \\  =&\theta(u,u)^{n+1}(x) - \sum\limits_{k=0}^{n-1}\tbinom{2n}{2k+1}\theta(u,u)^{n-k}\theta(x,u)\theta(u,u)^k(u) 
        \\ &+ \sum\limits_{k=0}^{n-1}\tbinom{2n}{2k}\theta(u,u)^{k+1}\theta(u,x)\theta(u,u)^{n-1-k}(u) -2\sum\limits_{k=0}^n\tbinom{2n}{2k}\theta(u,u)^k\theta(x,u)\theta(u,u)^{n-k}(u)
        \\& +2 \sum\limits_{k=0}^{n-1}\tbinom{2n}{2k+1}\theta(u,u)^{n-k}\theta(u,x)\theta(u,u)^k(u)  + \sum\limits_{k=0}^n\tbinom{2n}{2k}\theta(u,u)^{n-k}\theta(u,x)\theta(u,u)^k(u)
        \\& - \sum\limits_{k=0}^{n-1}\tbinom{2n}{2k+1}\theta(u,u)^k\theta(x,u)\theta(u,u)^{n-k}(u)
        \\ =&  \theta(u,u)^{n+1}(x) + \sum\limits_{k=1}^{n}\tbinom{2n}{2k-2}\theta(u,u)^{k}\theta(u,x)\theta(u,u)^{n-k}(u) 
        \\ & + \sum\limits_{k=0}^{n-1}(\tbinom{2n+1}{2k+1}+\tbinom{2n}{2k+1})\theta(u,u)^{n-k}\theta(u,x)\theta(u,u)^k(u) + \theta(u,x)\theta(u,u)^n 
        \\ &- \sum\limits_{k=1}^n\tbinom{2n}{2n-2k+1}\theta(u,u)^k\theta(x,u)\theta(u,u)^{n-k}(u) - \sum\limits_{k=0}^{n-1}\tbinom{2n+1}{2k+1}\theta(u,u)^k\theta(x,u)\theta(u,u)^{n-k}(u)
        \\ &- \theta(u,u)^n\theta(x,u) - \sum\limits_{k=0}^n\tbinom{2n}{2k}\theta(u,u)^k\theta(x,u)\theta(u,u)^{n-k}(u)
        \\ =&\theta(u,u)^{n+1}(x)+ \sum\limits_{k=0}^{n-1}(\tbinom{2n+1}{2k+1}+\tbinom{2n}{2k+1}+\tbinom{2n}{2k+2} )\theta(u,u)^{n-k}\theta(u,x)\theta(u,u)^k(u)\\ & + \theta(u,x)\theta(u,u)^n(u) 
         - \sum\limits_{k=1}^{n-1}(\tbinom{2n}{2k-1}+\tbinom{2n+1}{2k+1}+\tbinom{2n}{2k})\theta(u,u)^k\theta(x,u)\theta(u,u)^{n-k}(u)\\& - \theta(u,u)^n\theta(x,u)(u) 
         - \tbinom{2n+1}{1}\theta(x,u)\theta(u,u)^n(u)-\tbinom{2n}{1}\theta(u,u)^n\theta(x,u)(u) \\&- \theta(x,u)\theta(u,u)^n(u) - \theta(u,u)^n\theta(x,u)(u)
        \\ =& \theta(u,u)^{n+1}(x)+\sum\limits_{k=0}^{n}\tbinom{2n+2}{2k}\theta(u,u)^k\theta(u,x)\theta(u,u)^{n-k}(u) \\ &- \sum\limits_{k=0}^{n-1}\tbinom{2n+2}{2k+1}\theta(u,u)^{n-k}\theta(x,u)\theta(u,u)^k(u)
    \end{align*} The result follows. 
\end{proof}

\begin{cor}\label{c5.8}
     Let $(T,\left\{ \cdot,\cdot,\cdot\right\}) $ be a pre-Lie triple system and  $\theta(x,y)(z)= \left\{ z,x,y\right\} $ be the representation associated with the Lie triple system $(T,[\cdot,\cdot,\cdot]_C) $. Then,  for all $x,u\in T$,  we have \begin{center}
         $(x,u,\ldots,u)_C= \theta(u,u)^{\frac{p-1}{2}}(x) + \sum\limits_{k=0}^{\frac{p-1}{2}-1}\theta(u,u)^k\theta(u,x)\theta(u,u)^{\frac{p-1}{2}-1-k}(u) + \sum\limits_{k=0}^{\frac{p-1}{2}-1}\theta(u,u)^{\frac{p-1}{2}-1-k}\theta(x,u)\theta(u,u)^k(u) ,$
     \end{center} where $u$ occurs $p-1$ times.
\end{cor}

\begin{proof}
    We apply  \hyperref[p5.7]{Proposition~\ref*{p5.7}}, with $n= \frac{p-1}{2}$, and use the number theoretic result $$\tbinom{p-1}{i}\equiv (-1)^i[p], $$ for all $ 0\leq i\leq p-1$.
\end{proof}

We have the following result.

\begin{thm}\label{T5.9}
    Let $(T\left\{ \cdot,\cdot,\cdot\right\})$ be a pre-Lie triple system and $\theta(x,y)(z)= \left\{ z,x,y\right\} $ be the representation associated with the Lie triple system $(T,[\cdot,\cdot,\cdot]_C) $. We define for $y\in T$, $$y^{\left\{p\right\}}= \theta(y,y)^{\frac{p-1}{2}}(y).  $$ Then, for all $x,y\in T$, we have $$(x+y)^{\left\{p\right\}}= x^{\left\{p\right\}} + y^{\left\{p\right\}} + \sum\limits_{i=1}^{p-1}s_i(x,y), $$ where $s_i(x,y)$ are the coefficients defined in the Lie triple system $(T,[\cdot,\cdot,\cdot]_C) $.
\end{thm}

\begin{proof}
    We adapt  Jacobson's identities proof from  associative algebras case. We write $Z = (\lambda x+y)^{\left\{p\right\}}= u^{\left\{p\right\}} $ as a polynomial on $\lambda $: $$Z = \lambda^p x^{\left\{p\right\}} + \sum\limits_{i=1}^{p-1}\lambda^i\mu_i(x,y)+ y^{\left\{p\right\}}, $$ with coefficients $\mu_i(x,y)\in T $. We can derive this equality with respect to $\lambda$ and then obtain the right hand side $$\frac{\partial Z}{\partial \lambda}= \sum\limits_{i=1}^{p-1}i \lambda^{i-1}\mu_i(x,y) $$ and for the left hand side, by Leibniz rule, $$\frac{\partial Z}{\partial \lambda}= \theta(u,u)^{\frac{p-1}{2}}(x)+ \sum\limits_{k=0}^{\frac{p-1}{2}-1}\theta(u,u)^k(\theta(u,x)+\theta(x,u))\theta(u,u)^{\frac{p-1}{2}-1-k}(u). $$ By \hyperref[c5.8]{Corollary~\ref*{c5.8}}, we have $$\frac{\partial Z}{\partial \lambda}= (x,u,\ldots,u)_C,$$ where $u$ occurs $p-1$ times. It shows that the $\mu_i(x,y) $ are exactly the $s_i(x,y)$, finishing  the proof.
\end{proof}

This motivates the following definition. We  define trivially restricted pre-Lie triple systems in a similar way as for pre-Lie algebras.

\begin{defi}
    Let $(T,\left\{ \cdot,\cdot,\cdot\right\}) $ be a pre-Lie triple system. We say that $(T,\left\{ \cdot,\cdot,\cdot\right\}) $ is trivially restricted if the following identities hold: \begin{enumerate}
        \item $\left\{ x,y^{\left\{p\right\}},z\right\}=\left\{\cdots \left\{\left\{\left\{ x,y,y\right\},y,y \right\}\cdots\right\},y,z\right\} $,
        \item $\left\{ x,y,z^{\left\{p\right\}} \right\}=\left\{\cdots \left\{\left\{\left\{ x,y,z\right\},z,z \right\}\cdots\right\},z,z\right\}, $
    \end{enumerate} for all $x,y,z\in T$, and where the $p$-map is defined according to Theorem \ref{T5.9}.
\end{defi}

\begin{ex}
    The pre-Lie triple system $A$ defined in  \hyperref[ex4.7]{Example~\ref*{ex4.7}} is trivially restricted. Indeed, any iteration of the bracket is zero. 
\end{ex}

\begin{prop}
    Let $(A,\circ)$ be a trivially restricted pre-Lie algebra. Then, $(A,\left\{ \cdot,\cdot,\cdot\right\})$ is a trivially restricted pre-Lie triple system.
\end{prop}

\begin{proof}
We first notice that for $y\in A$, $y^{\circ.p}= y^{{\left\{p\right\}}}$. Let $x,y,z\in A $. $$\left\{ x,y^{\left\{p\right\}},z\right\}= z\circ(y^{\circ.p} \circ x)=z\circ (y\circ (y\circ (\cdots (y\circ x))))= \left\{\cdots \left\{\left\{\left\{ x,y,y\right\},y,y \right\}\cdots\right\},y,z\right\},   $$   
$$\left\{ x,y,z^{\left\{p\right\}} \right\}= z^{\circ.p}\circ(y\circ x)=(z\circ (z\circ (\cdots (z\circ(y\circ x)))))=\left\{\cdots \left\{\left\{\left\{ x,y,z\right\},z,z \right\}\cdots\right\},z,z\right\}. $$
\end{proof}

\begin{prop}\label{p5.12}\cite{BM26}
    Let $T$ be a Lie triple system and $\theta:T\times T\to End(V)$ be a representation of $T$. Then for all $x,y,z\in T $, we have \begin{center}
        $\theta([[\dots[y,x,x]\dots],x,x],z)= \theta(x,z)\sum\limits_{k=1}^n \binom{2n}{2k}\theta(x,x)^{k-1}\theta(y,x)\theta(x,x)^{n-k} -\theta(x,z) \sum\limits_{k=0}^{n-1}\binom{2n}{2k+1}\theta(x,x)^{n-k-1}\theta(x,y)\theta(x,x)^k + \theta(y,z)\theta(x,x)^n $,
    \end{center} where $x$ occurs $2n$ times in the bracket.
\end{prop}

\begin{thm}\label{thm5.13}
    Let $(T,\left\{ \cdot,\cdot,\cdot\right\})$ be a trivially restricted pre-Lie triple system. Then, $(T,[\cdot,\cdot,\cdot]_C) $ is a restricted Lie triple system with $p$-map $x^{[p]}=x^{\left\{p\right\}} $. Moreover, the bilinear map $\theta(x,y)(z)= \left\{ z,x,y\right\} $ is a restricted representation of the restricted Lie triple system $(T,[\cdot,\cdot,\cdot]_C) $.
\end{thm}

\begin{proof}
 Items $1.$ and $2.$ in \hyperref[D5.1]{Definition~\ref*{D5.1}} are fulfilled. Item $1.$ is clear and Item $2.$ derives from  \hyperref[T5.9]{Theorem~\ref*{T5.9}}. Item $3.$ remains to be verified. Since $T$ is trivially restricted so \begin{equation}\label{eq6} \left\{ x,y^{\left\{p\right\}},z\right\}=\left\{\cdots \left\{\left\{\left\{ x,y,y\right\},y,y \right\}\cdots\right\},y,z\right\} \end{equation} and \begin{equation}\label{eq7} \left\{ z,x,y^{\left\{p\right\}} \right\}=\left\{\cdots \left\{\left\{\left\{ z,x,y\right\},y,y \right\}\cdots\right\},y,y\right\} \end{equation} for all $x,y,z\in T$. 
     
     Let $x,y,z \in T$.
     By  \hyperref[c5.8]{Corollary~\ref*{c5.8}},  \hyperref[p5.12]{Proposition~\ref*{p5.12}} and  \hyperref[p5.6]{Proposition~\ref*{p5.6}}, we have $(x,y,\ldots,y,z)_C $, where $y$ occurs $p$ times given by
     \begin{align*} (x,y,&\ldots,y,z)_C= [(x,y,\ldots,y)_C,y,z]_C 
     \\ =&  \left\{(x,y,\ldots,y)_C,y,z \right\} - \left\{y,(x,y,\ldots,y)_C,z \right\} + \left\{z,y,(x,y,\ldots,y)_C \right\}- \left\{z,(x,y,\ldots,y)_C,y \right\}
     \\ =& \theta(y,z)((x,y,\ldots,y)_C) - \theta((x,y,\ldots,y)_C,z)(y) + \theta(y,(x,y,\ldots,y)_C)(z) - \theta((x,y,\ldots,y)_C,y)(z) 
     \\ =& \theta(y,z)(\theta(y,y)^{\frac{p-1}{2}}(x) + \sum\limits_{k=0}^{\frac{p-1}{2}-1}\theta(y,y)^k(\theta(y,x)+\theta(x,y))\theta(y,y)^{\frac{p-1}{2}-1-k}(y) )
     \\ & - \theta(y,z)\sum\limits_{k=1}^{\frac{p-1}{2}}\theta(y,y)^{k-1}\theta(x,y)\theta(y,y)^{\frac{p-1}{2}-k}(y) -\theta(y,z) \sum\limits_{k=0}^{\frac{p-1}{2}-1}\theta(y,y)^{\frac{p-1}{2}-k-1}\theta(y,x)\theta(y,y)^k(y)
     \\ &- \theta(x,z)\theta(y,y)^{\frac{p-1}{2}}(y) + \sum\limits_{k=0}^{\frac{p-1}{2}}\theta(y,y)^{\frac{p-1}{2}-k}\theta(y,x)\theta(y,y)^k(z) + \sum\limits_{k=0}^{\frac{p-1}{2}-1}\theta(y,y)^k\theta(x,y)\theta(y,y)^{\frac{p-1}{2}-k}(z)
     \\ & - \sum\limits_{k=0}^{\frac{p-1}{2}}\theta(y,y)^k\theta(x,y)\theta(y,y)^{\frac{p-1}{2}-k}(z) - \sum\limits_{k=0}^{\frac{p-1}{2}-1}\theta(y,y)^{\frac{p-1}{2}-k}\theta(y,x)\theta(y,y)^k(z)
     \\ = & \theta(y,z)\theta(y,y)^{\frac{p-1}{2}}(x) - \theta(x,z)\theta(y,y)^{\frac{p-1}{2}}(y) + \theta(y,x)\theta(y,y)^{\frac{p-1}{2}}(z) - \theta(y,y)^{\frac{p-1}{2}}\theta(x,y)(z).   \end{align*} In the other hand, because $T$ is trivially restricted, we have  \begin{align*}
         [x,y^{\left\{p\right\}},z]_C =& \left\{ x,y^{\left\{p\right\}},z\right\} -\left\{y^{\left\{p\right\}},x,z\right\}+\left\{z,y^{\left\{p\right\}},x\right\}-\left\{ z,x,y^{\left\{p\right\}}\right\}
          \\  = &  \left\{\cdots \left\{\left\{\left\{ x,y,y\right\},y,y \right\}\cdots\right\},y,z\right\} - \left\{y^{\left\{p\right\}},x,z\right\}
          \\& + \left\{\cdots \left\{\left\{\left\{ z,y,y\right\},y,y \right\}\cdots\right\},y,x\right\} - \left\{\cdots \left\{\left\{\left\{ z,x,y\right\},y,y \right\}\cdots\right\},y,y\right\}
          \\ = & \theta(y,z)\theta(y,y)^{\frac{p-1}{2}}(x) - \theta(x,z)\theta(y,y)^{\frac{p-1}{2}}(y) + \theta(y,x)\theta(y,y)^{\frac{p-1}{2}}(z) - \theta(y,y)^{\frac{p-1}{2}}\theta(x,y)(z) 
          \\ = & (x,y,\ldots,y,z)_C.
     \end{align*} The result follows. The fact that then the bilinear map $\theta$ is a restricted representation is an immediate consequence of the definition of  trivially restricted pre-Lie triple system.  Equations \eqref{eq6} and \eqref{eq7} mean that  \begin{center}
    $\theta(y^{[p]},z)=\theta(y,z) \circ \theta(y,y)^{\frac{p-1}{2}} $ 
    \end{center} and 
    \begin{center} $\theta(x,y^{[p]})= \theta(y,y)^{\frac{p-1}{2}}\circ\theta(x,y) $.
    \end{center} 
\end{proof}

Now, following Dokas, we provide here a more general definition of restricted pre-Lie triple systems.

\begin{defi}\label{D5.14}
    A restricted pre-Lie triple system is a pre-Lie triple system $T$ equipped with a $p$-map $(\cdot)^{[p]}:T\to T $ such that the following identities are satisfied: \begin{enumerate}
        \item $(\alpha x)^{[p]}= \alpha^px^{[p]} $,
        \item $\left\{ x,y^{[p]},z\right\}=\left\{\cdots \left\{\left\{\left\{ x,y,y\right\},y,y \right\}\cdots\right\},y,z\right\} $,
        \item $\left\{ x,y,z^{[p]} \right\}=\left\{\cdots \left\{\left\{\left\{ x,y,z\right\},z,z \right\}\cdots\right\},z,z\right\} $
        \item $\left\{ z,y^{[p]},x\right\}-\left\{z,x,y^{[p]} \right\}+\left\{ x,y^{[p]} ,z\right\}-\left\{y^{[p]},x,z \right\}= (x,y,\ldots,y,z)_C $,
        \item $(x+y)^{[p]}= x^{[p]}+y^{[p]}+ \sum\limits_{i=1}^{p-1} s_i(x,y) $,    
    \end{enumerate} for all $x,y,z\in T$, where $s_i(x,y)$ are the coefficients in the Lie triple system $(T,[\cdot,\cdot,\cdot]_C) $.
\end{defi}
Naturally, trivially restricted pre-Lie triple systems are restricted pre-Lie triple systems.
\begin{prop}
    Let $(T,\left\{ \cdot,\cdot,\cdot\right\})$ be a trivially restricted pre-Lie triple system. Then $(T,\left\{ \cdot,\cdot,\cdot\right\},(\cdot)^{\left\{p\right\}}) $ is a restricted pre-Lie triple system.
\end{prop}

\begin{proof}
The result follows from \hyperref[thm5.13]{Theorem~\ref*{thm5.13}}.
\end{proof}

\begin{prop}
    Let $T$ be a restricted pre-Lie triple system. Then $(T,[\cdot,\cdot,\cdot])_C $ is a restricted Lie triple system with the same $p$-map. 
\end{prop}

\begin{proof}
    It is straightforward.
\end{proof}
Next, we show that we can obtain restricted pre-Lie triple systems from restricted pre-Lie algebras.
\begin{prop}
    Let $(A,\circ,(\cdot)^{[p]})$ be a restricted pre-Lie algebra. Then $(A,\left\{ \cdot,\cdot,\cdot\right\},(\cdot)^{[p]} )$ is a restricted pre-Lie triple system.
\end{prop}

\begin{proof}
    Let $(A,\circ) $ be a restricted pre-Lie algebra, then the induced Lie algebra is restricted. The first three Items of \hyperref[D5.14]{Definition~\ref*{D5.14}} of restricted pre-Lie triple systems follow immediately from  the definition of $\left\{ \cdot,\cdot,\cdot\right\} $ and the fact that $(A,\circ) $ is a restricted pre-Lie algebra. The last two Items follow from the fact that the induced Lie algebra is restricted so as the induced Lie triple system and the fact that the diagram is commutative.
\end{proof}

The diagram \[\begin{tikzcd}[row sep=large, column sep=large]
	{ \text{$p$-pre-Lie}} && \text{$p$-Lie} \\
	{\text{$p$-pre-LTS}} && \text{$p$-LTS}
	\arrow["{[x,y]=x\circ y - y\circ x}", from=1-1, to=1-3]
	\arrow["{\left\{ x,y,z\right\}= z\circ (y\circ x)}"', from=1-1, to=2-1]
	\arrow["{[x,y,z]=[[x,y],z]}", from=1-3, to=2-3]
	\arrow["{[x,y,x]_C}"', from=2-1, to=2-3]
\end{tikzcd}\] is commutative, where the $p$-map is the same for all structures.


\subsection{Restricted Rota-Baxter on restricted Lie triple systems and   restricted pre-Lie triple systems}

\begin{prop}\label{p5.20}
    Let $(T,[\cdot,\cdot,\cdot],(\cdot)^{[p]})$ be a restricted Lie triple system with a restricted Rota-Baxter operator $P$. We  define on $T$ a structure of restricted pre-Lie triple system by setting $$\left\{ x,y,z\right\}= [x,Py,Pz] $$ and $$x^{[p]'}= (x,Px,\ldots, Px). $$ Moreover, $P$ is then a restricted morphism of restricted Lie triple systems, $P: (T,[\cdot,\cdot,\cdot]_C,(\cdot)^{[p]'})\to (T,[\cdot,\cdot,\cdot],(\cdot)^{[p]}) $.
\end{prop}

\begin{proof}
    First, we notice that $y^{[p]'}=y^{\left\{p\right\}} $.  So Jacobson identities are already satisfied. Let $x,y,z\in T$, we have \begin{align*}
        \left\{ x,y^{\left\{p\right\}},z\right\}=& [x,P(y^{\left\{p\right\}}),P(z)] \\ =&[x,P(y)^{[p]},P(z)]\\ =& (x,Py,\ldots,Py,Pz) \\ = &\left\{\cdots \left\{\left\{\left\{ x,y,y\right\},y,y \right\}\cdots\right\},y,z\right\}. \end{align*} Similarly, \begin{align*}
            \left\{ x,y,z^{\left\{p\right\}}\right\}=& [x,P(y),P(z^{\left\{p\right\}})] \\ =&[x,P(y),P(z)^{[p]}]\\ =& (x,Py,Pz,\ldots,Pz) \\ = &\left\{\cdots \left\{\left\{\left\{ x,y,z\right\},z,z \right\}\cdots\right\},z,z\right\}. \end{align*}
        Thus, $(T,\left\{\cdot,\cdot,\cdot \right\})$ is a trivially restricted pre-Lie triple system and the result follows.
\end{proof}

We have again the same situation as in the non restricted case. From a restricted Lie triple system $(T,[\cdot,\cdot,\cdot])$, with a restricted Rota-Baxter operator, we define a restricted pre-Lie triple system $(T,\left\{ \cdot,\cdot,\cdot\right\}) $, and in addition the map $P$ is a restricted morphism of restricted Lie triple systems $P:(T,[\cdot,\cdot,\cdot]_C)\to (T,[\cdot,\cdot,\cdot]) $.
We summarize the results in the following diagram. 
\[\begin{tikzcd}[row sep=large, column sep=large]
	{(T,[\cdot,\cdot,\cdot],(\cdot)^{[p]})} & {(T,\left\{\cdot,\cdot,\cdot \right\},(\cdot)^{[p]')}} \\
	& {(L,[\cdot,\cdot,\cdot]_C,(\cdot)^{[p]'})}
	\arrow["{[x,Py,Pz]}", from=1-1, to=1-2]
	\arrow[from=1-2, to=2-2]
	\arrow["P", from=2-2, to=1-1]
\end{tikzcd}\]

\section{Relative Rota-Baxter operators}\label{s6}
We aim in this section to consider a more general situation for restricted Rota-Baxter operators. We extend to restricted case, the notion of Relative Rota-Baxter operator, which are also called $\mathcal{O}$-operators and dealing with an algebra and a representation. We also show some relevant results. 
We  adapt the definition of  restricted Rota-Baxter operators in restricted Lie algebras and Lie triple systems in the case of relative Rota-Baxter with respect to restricted representations. Then \hyperref[d5.6]{Definition~\ref*{d5.6}}, \hyperref[p5.8]{Proposition~\ref*{p5.8}} and \hyperref[p5.20]{Proposition~\ref*{p5.20}} are generalized as follows.

\begin{defi}
    Let $T$ be a Lie triple system and $(V,\theta)$ be a representation of $T$. A linear map $P:V\to T $ is called a relative Rota-Baxter operator with respect to a representation $(V,\theta)$ if \begin{center}
        $[P u,P v,P w]=P (D(P u,P v)(w)+\theta(P v,P w)(u)- \theta(P u,P w)(v)) $ for all $u,v,w\in V$.
    \end{center}
\end{defi}

\begin{ex}
    Let $T$ be a Lie triple system. Then a relative Rota-Baxter operator with respect to the adjoint representation of $T$ is a Rota-Baxter operator of weight $0$ on $T$.
\end{ex}

\begin{prop}
    Let $T$ be a Lie triple system and $(V,\theta)$ be a representation of $T$. A linear map $P:V\to T$ is a relative Rota Baxter operator on $T$ with respect to $(V,\theta) $ if and only if its graph $Gr(P)= \left\{ (P u,u), u\in V\right\} $ is a subsystem of $T\oplus V $.
\end{prop}

\begin{prop}[\cite{M21}]
    Let $(A,[\cdot,\cdot,\cdot]) $ be a Lie triple system and $(V,\theta) $ a representation. Suppose that the linear map $P:V\to A$ is a relative Rota-Baxter operator with respect  to $(V,\theta)$. Then, there exists a pre-Lie triple system structure on $V$ given by \begin{center}
        $\left\{ u,v,w\right\} = \theta(Pv,Pw)(u)$, $\forall u,v,w\in V $.
    \end{center} 
\end{prop}

In the sequel, we define restricted relative Rota-Baxter operators on a restricted Lie triple system, we want to keep previous proposition true in the restricted case.

\begin{defi}
    Let $T$ be a restricted Lie triple system and $(V,\theta) $ be a restricted representation of $T$. A linear map $P:V\to T $ is called a restricted relative Rota-Baxter operator on $T$ if it is a relative Rota-Baxter operator of the Lie triple system $T$ and if, in addition, \begin{center}
        $P (u)^{[p]}=P(\theta(P u,P u)^{\frac{p-1}{2}}(u)) $ for all $u\in V$.
    \end{center} 
\end{defi}

\begin{prop}
    Let $T$ be a restricted Lie triple system and $(V,\theta)$ be a restricted representation of $T$. A linear map $P:V\to T$ is a restricted relative Rota-Baxter operator on $T$ with respect to $(V,\theta) $ if and only if its graph $Gr(P)= \left\{ (P u,u), u\in V\right\} $ is a restricted subsystem of $T\oplus V $, where the semi-direct product $T\oplus V $ has a $p$-map defined in  \hyperref[thm5.7]{Theorem~\ref*{thm5.7}} for any basis of $T\oplus V $.
\end{prop}

\begin{proof}
    Let $(Pu_i,u_i) $ be a basis of $Gr(P)$. We can complete this basis into a basis $(x_i,u_i) $ of $T\oplus V $. We have a $p$-map on $T\oplus V$ defined on the basis by $(x_i,u_i)^{[p]}=(x_i^{[p]},\theta(x_i,x_i)^{\frac{p-1}{2}}(u_i)) $. The graph  $Gr(P) $ is a subsystem of $T\oplus V$. Then it is sufficient  to show that $Gr(P)$ is stable under the $p$-map on $T\oplus V$. For that, it is enough to show that for $(Px_i,x_i) $ elements of the basis of $Gr(P)$, $(Pu_i,u_i)^{[p]}=(P(u_i)^{[p]}, \theta(Pu_i,Pu_i)^{\frac{p-1}{2}}(u_i))  \in Gr(P) $. It is the case from the definition of a restricted relative Rota-Baxter operator. Conversely, assume that $Gr(P) $ is closed under any $p$-map defined with a basis $(x_i,u_i) $ of $T\oplus V $. Let $u\in V$. We can complete $(Pu,u) $ into a basis of $T\oplus V $. Then there exists a $p$-map defined on $(Pu,u) $ by $(Pu,u)^{[p]}=(P(u)^{[p]}, \theta(Pu,Pu)^{\frac{p-1}{2}}(u))  $. Also $Gr(P)$ is closed under $(\cdot)^{[p]} $, so $P(u)^{[p]}=P(\theta(Pu,Pu)^{\frac{p-1}{2}}(u)) $ and $P$ is a restricted relative Rota-Baxter operator.
\end{proof}

Since $V$ and $Gr(P)$ are isomorphic, we can define on $V$ a structure of restricted triple system by setting $[u,v,w]=D(P u,P v)(w)+\theta(P v,P w)(u)- \theta(P u,P w)(v) $ and $u^{[p]}= \theta(P u,P u)^{\frac{p-1}{2}}(u) $. In this case, $P:V\to T $ is a morphism of restricted Lie triple systems.

Finally, we have the following result.
\begin{prop}
    Let $(T,[\cdot,\cdot,\cdot], (\cdot)^{[p]})$ be a restricted Lie triple system with a restricted representation $(V,\theta) $ and a restricted relative Rota-Baxter operator $P$ with respect to $(V,\theta)$. We  define on $V$ a structure of restricted pre-Lie triple system by defining for all $u,v,w\in V$ $$\left\{ u,v,w\right\}= \theta(Pv,Pw)(u) $$ and $$u^{[p]'}= \theta(Pu,Pu)^{\frac{p-1}{2}}(u) . $$ Moreover, $P$ is then a restricted morphism of restricted Lie triple systems $$P: (V,[\cdot,\cdot,\cdot]_C,(\cdot)^{[p]'})\to (T,[\cdot,\cdot,\cdot],(\cdot)^{[p]}). $$
\end{prop}

\begin{proof}
    First, we notice that $u^{[p]'}=u^{\left\{p\right\}} $. So  Jacobson identities are already verified. Let $u,v,w\in V$, \begin{align*}
        \left\{ u,v^{\left\{p\right\}},w\right\}=& \theta(P(v^{\left\{p \right\}}),Pw)(u) \\ =&\theta(P(v)^{[p]},Pw)(u)\\ =& \theta(Pv,Pw)\circ \theta(Pv,Pv)^{\frac{p-1}{2}}(u) \\ = &\left\{\cdots \left\{\left\{\left\{ u,v,v\right\},v,v \right\}\cdots\right\},v,w\right\}. \end{align*} Similarly, \begin{align*}
            \left\{ u,v,w^{\left\{p\right\}}\right\}=& \theta(P(v),P(w^{\left\{p\right\}})(u) \\ =&\theta(P(v),P(w)^{[p]})(u)\\ =& \theta(Pw,Pw)^{\frac{p-1}{2}}\circ\theta(Pv,Pw)(u) \\ = &\left\{\cdots \left\{\left\{\left\{ u,v,w\right\},w,w \right\}\cdots\right\},w,w\right\}. \end{align*}
        The result follows.
\end{proof}


\end{document}